\documentclass[12pt]{amsart}
\usepackage{graphicx,caption,subcaption}
\usepackage{tikz}        
\usepackage{amssymb}
\usepackage{epstopdf}
\usepackage[mathscr]{eucal}
\usepackage{amsmath,amssymb,amscd}
\usepackage{color}
\usepackage{epsfig}
\usepackage{hyperref}
\newtheorem{theorem}{Theorem}
\newtheorem{lemma}{Lemma}

\newtheorem{definition}{Definition}

\newtheorem{proposition}{Proposition}
\theoremstyle{definition}
\newtheorem{remark}{Remark}

\def \mb{\mathbb}

\def \bf{\mathbf}

\def \mr{\mathrm}

\def \R{\mb R}                 

\def \a{\alpha}         
  
\def \lm{\lambda}     

\def\v{{\bf v}}

\newcommand {\tQ} {\tilde{Q}}

\newcommand {\p} {\mathbf{p}}

\newcommand {\sk} {\mathfrak{s}}
\newcommand {\cH} {\mathcal{H}}
\newcommand {\cW} {\mathcal{W}}
\newcommand {\kS} {\mathfrak{S}}
\newcommand {\bI} {\mathbf{I}}
\newcommand {\e} {\epsilon}
\newcommand {\de} {\delta}

\newcommand{\lec}{\lesssim}
\newcommand{\gec}{\gtrsim}

\newcommand{\EQ}[1]{\begin{equation}\begin{split} #1 \end{split}\end{equation}}

\def \and{\mbox{and}}
\usepackage[top=4cm, bottom=4cm, left=3.5cm, right=3.5cm]{geometry}
\title{Global existence and singularity of the Hill's type lunar problem with strong potential}
\date{\today}

\begin{document}

\maketitle

\markboth{Yanxia Deng, Slim Ibrahim}{Hill's type lunar problem with strong potential}
\author{\begin{center} 
Yanxia Deng ~~\footnote{School of Mathematics (Zhuhai), Sun Yat-sen University, Zhuhai, Guangdong, China \quad dengyx53@mail.sysu.edu.cn}, 
Slim Ibrahim\footnote{Department of Mathematics and Statistics, University of Victoria, Victoria, BC, Canada \quad ibrahims@uvic.ca}\\

\end{center}}
\begin{abstract}
We characterize the fate of the solutions of Hill's type lunar problem using the ideas of ground states from PDE. In particular, the relative equilibrium will be defined as the ground state, which satisfies some crucial energetic variational properties in our analysis. We study the dynamics of the solutions below, at, and (slightly) above the ground state energy threshold. 
\end{abstract}

\section{Introduction.}
The three-body problem is a prototypical case in celestial mechanics. The system Sun-Earth-Moon can be considered as a typical example of the three-body problem. Using heuristic arguments about the relative size of various physical constants, Hill was able to give the equations for the motion of the moon as an approximation from the general three-body problem.  The Hill's lunar problem can be derived from the general three-body problem using symplectic scaling method \cite{MeSc82}. We give a derivation of the main problem for homogeneous potential in the last section of the paper. A popular description of the Hill's equations is to consider the motion of an infinitesimal body (the moon) which is attracted to a body (the earth) fixed at the origin. The infinitesimal body moves in a rotating coordinate system which rotates so that the positive $x$-axis points towards an infinite body (the sun) which is infinitely far way. The ratio of the two infinite quantities is taken so that the gravitational attraction of the sun on the moon is finite.

In particular, if the position of the moon is given by $(x,y)$, the planar Hill's equation with homogenous gravitational potential is given by

 \begin{equation}\label{eq:hlp}\begin{cases}\ddot{x}-2\dot{y}&=-V_x\\\ddot{y}+2\dot{x}&=-V_y,\end{cases}\end{equation}
where \begin{equation}V(x,y)=-\frac{\a+2}{2}x^2-\frac{\a+2}{r^\alpha},\quad r=\sqrt{x^2+y^2}, \quad \alpha>0\end{equation} is known as the effective potential. When $\a=1$, this is the Newtonian Hill's Lunar Problem; when $\a\geq 2$, we shall call it the Hill's type lunar problem with \emph{strong potential}.

This vector field is well-defined everywhere except at the origin $(0,0)$, which is the position of the earth. By the existence and uniqueness theorem of ODE, given $q(0)=(x(0),y(0))\neq (0,0)$ and $\dot{q}(0)\in \R^2$, there exists a unique solution $q(t)$ defined on the interval $[0, T_\mathrm{max})$, where $T_\mathrm{max}$ is the maximal interval of existence. 

\begin{definition}
If $T_\mathrm{max}<\infty$, then the solution is said to experience a singularity at $T_\mathrm{max}$; otherwise, we say the solution exists globally.
\end{definition}

Since the ODE is locally Lipschitz in $r\neq0$, blow-up is possible only by approaching the unique singularity, namely the collision. In particular, for the Hill's equation, if $T_\mathrm{max}<\infty$, then \EQ{\lim\limits_{t\to T_\mathrm{max}} (x(t), y(t))= (0,0),}  that is, the singularity of the Hill's equation is due to finite time collision at the origin.

The Hill's equation admits the famous Jacobi integral which we shall refer to as the \emph{energy}, \begin{equation}E(x,y,\dot{x},\dot{y}):=\frac{1}{2}(\dot{x}^2+\dot{y}^2)+V(x,y).\end{equation} 

The effective potential $V(x,y)$ has exactly two critical points $L_1:=(-\alpha^{\frac{1}{\alpha+2}},0)$ and $L_2:=(\alpha^{\frac{1}{\alpha+2}},0)$, which are known as the Lagrange points. The projection of the four-dimensional phase space onto the configuration $(x,y)$ space is called the Hill's regions. 
\EQ{\cH_c:=\{(x,y) | E(x,y,\dot{x},\dot{y})=c\}=\{(x,y) | V(x,y)\leq c\}.} The boundaries of the Hill's regions are called \emph{zero velocity curves} because they are the locus in the configuration space where the kinetic energy vanishes. A contour plot for $V(x,y)$ in the $(x,y)$-plane could be seen in Figure \ref{fig:hlpVcon}.

\begin{figure}[ht]
	\centering
	\includegraphics[width=0.5\textwidth]{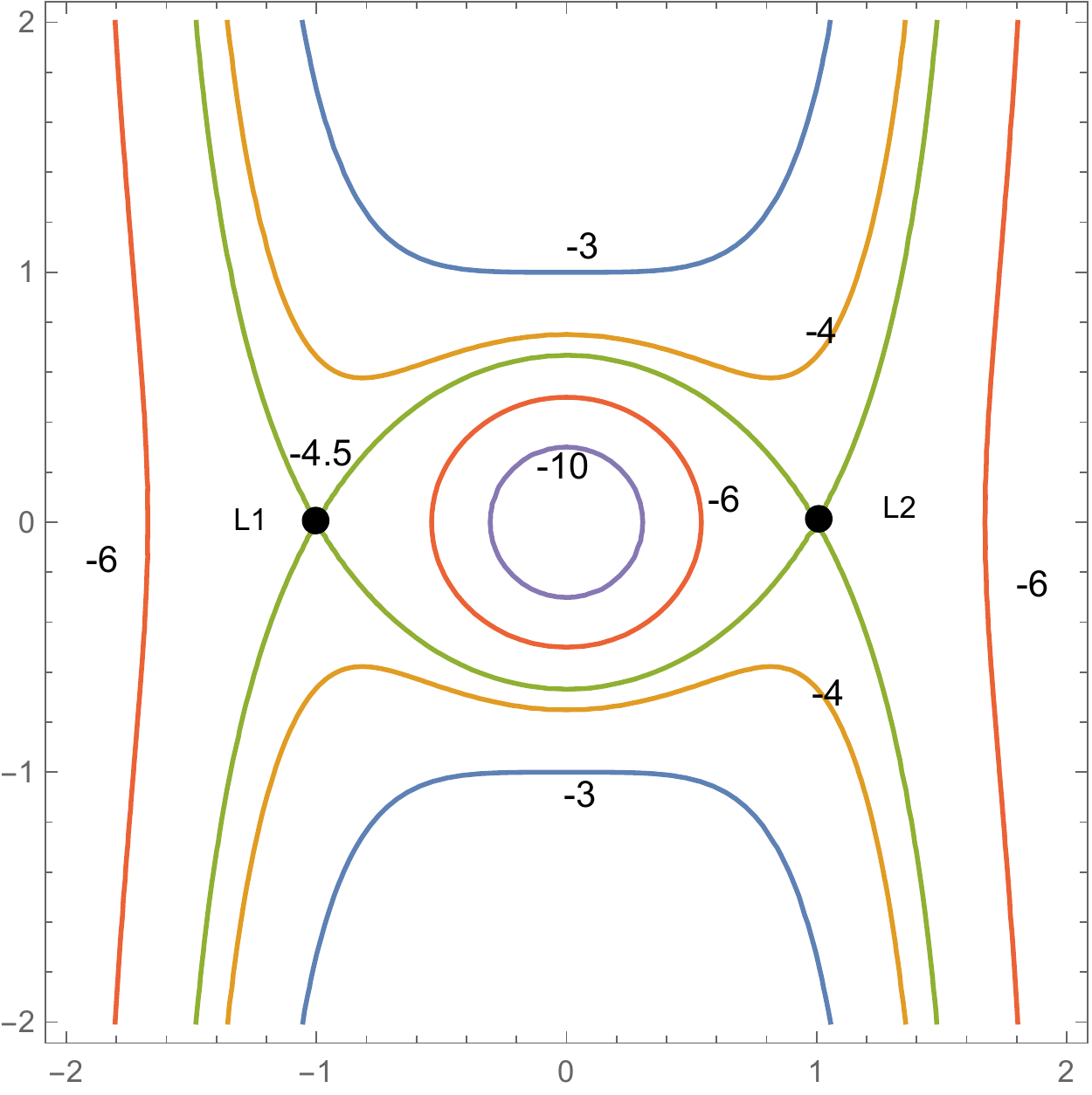}
	\caption{The contour plot of $V(x,y)$ with $\alpha=1$. $V(x,y)$ has two critical points $L_1:=(-\alpha^{\frac{1}{\alpha+2}},0)$ and $L_2:=(\alpha^{\frac{1}{\alpha+2}},0)$.}
\label{fig:hlpVcon}
\end{figure}

The structure of the Hill's regions depends on the value of the energy. There are four distinct cases regarding the shape of the Hill's regions:
\begin{itemize}
\item[(i)] $c<E^*$: both necks are closed, so orbits inside will remain bounded in the configuration space or collide with the origin.
\item[(ii)] $c=E^*$: the threshold case.
\item[(iii)] $E^*<c<0$: both necks are open, thus allowing orbits to enter the exterior region and escape from the system.
\item[(iv)] $c\geq 0$: motions over the entire configuration $(x,y)$ space is possible.  
\end{itemize}

In Figure \ref{fig:Hillrg} we present the structure of the first and third possible Hill's region for $\a=1$; all the other $\a>0$ have the same structure with varied values of $L_1, L_2$.

\begin{figure}[h]
\begin{subfigure}[b]{0.45\textwidth}
\centering
  \includegraphics[width=\textwidth]{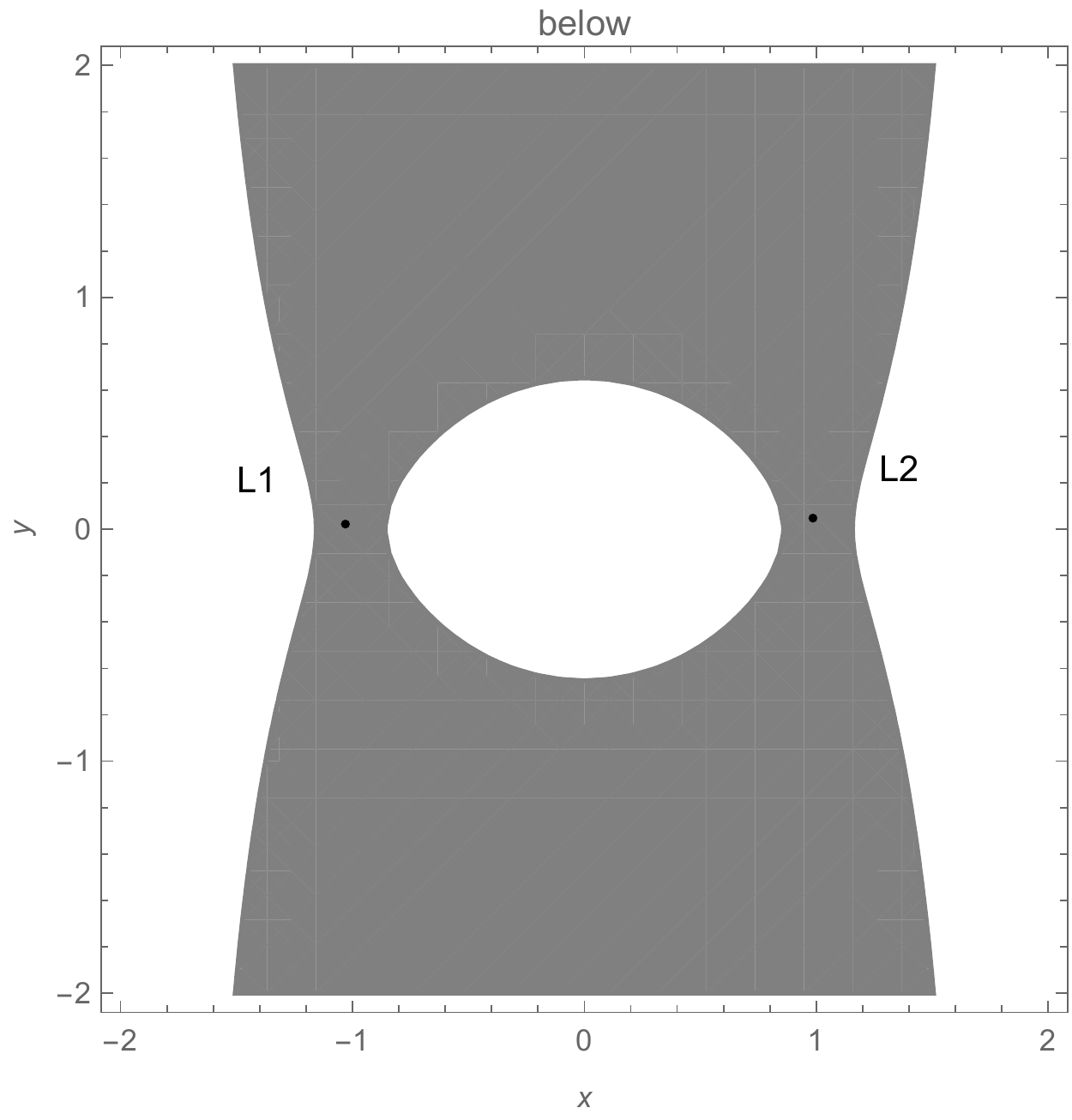}
  \end{subfigure}
\qquad\begin{subfigure}[b]{0.45\textwidth}
\centering
  \includegraphics[width=\textwidth]{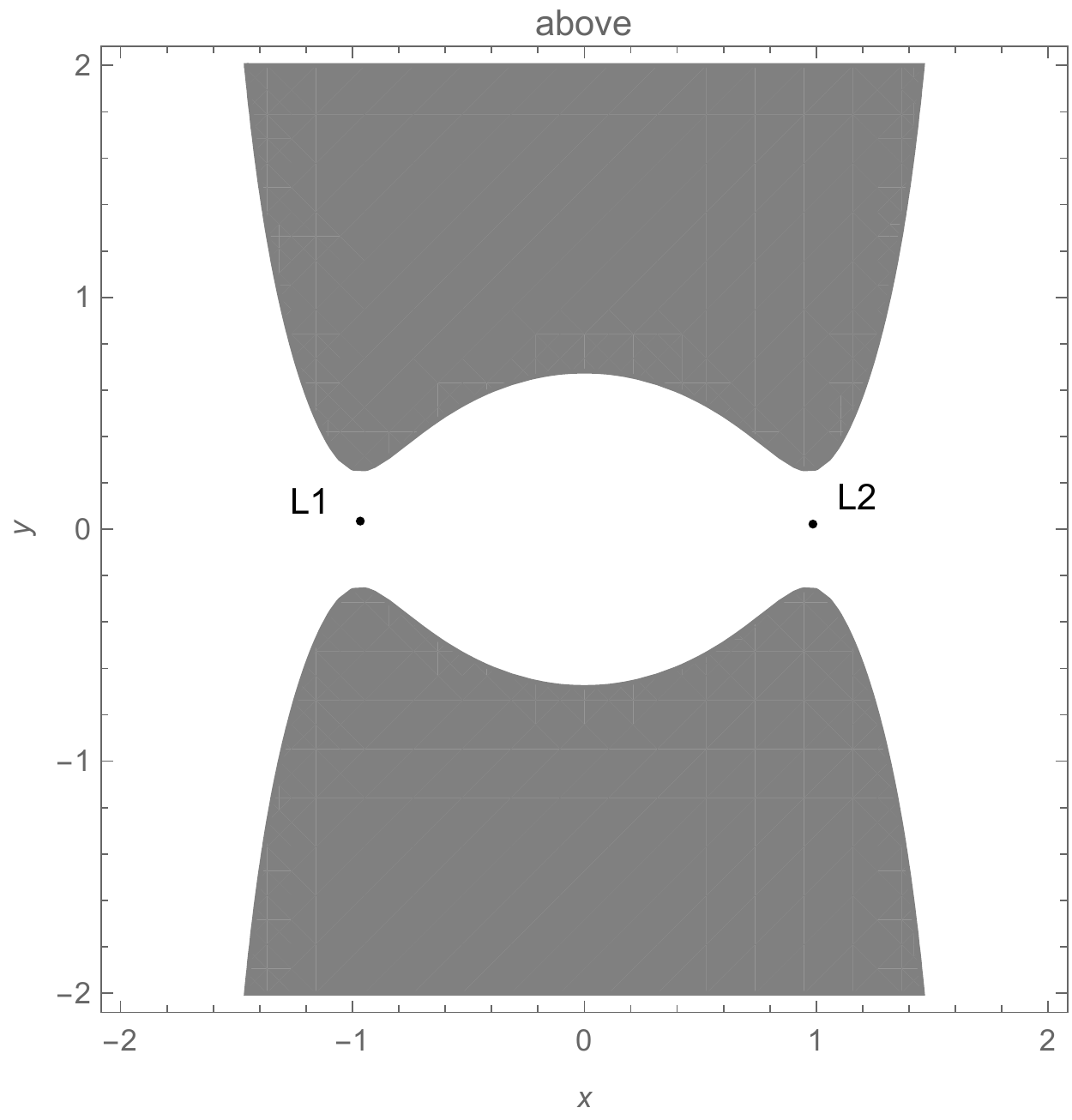}
  \end{subfigure}
\caption{Hill's regions $\cH_c$ when $\a=1$, $E^*=-4.5$. The white domains correspond to the Hill's regions, gray shaded domains indicate the energy forbidden regions. Left is below the ground state energy with $c=-4.6$; right is above the ground state energy with $c=-4.4$. }
\label{fig:Hillrg}
\end{figure}

Consequently, $\pm Q:=(\pm\alpha^{\frac{1}{\alpha+2}},0,0,0)$ are the only equilibria of (\ref{eq:hlp}). We will define $\pm Q$ to be the \emph{ground state}. In section \ref{sec:ground}, we will give a detailed analysis about the energetic variational properties satisfied by $\pm Q$ as well as some insights on why would we call them the ground states.  In particular, we will first define the \emph{ground state energy} $E^*$, which will be the minimum of the energy $E$ under the constraint $W=0$, where \EQ{W:=-xV_x-yV_y=(\a+2)(x^2-\frac{\alpha}{r^\alpha}).} Note that $E^*$ coincides with the energy of the equilibria. See Figure \ref{fig:hlpVW} for the graph of $V=E^*$ and $W=0$ in the $(x,y)$ plane. In particular, $W<0$ corresponds to the inner bounded region, and $W>0$ corresponds to the outer region. The rest of the paper is then devoted to study the dynamics of the solution of (\ref{eq:hlp}) with energy below, at, and (slightly) above the ground state energy threshold. Our main results are in the following theorems.

\begin{theorem}[Dichotomy below the ground state]
\label{thm:below}
For the Hill's lunar problem, consider the sets: \begin{equation}
\begin{split}
\cW_+&=\{\Gamma=(x, y, \dot{x}, \dot{y}) | E(\Gamma)<E^*, W(\Gamma) > 0\}\\
\cW_-&=\{\Gamma=(x, y, \dot{x}, \dot{y}) | E(\Gamma)<E^*, W(\Gamma) \leq 0\}
\end{split}
\end{equation}
then $\cW_+$ and $\cW_-$ are invariant. Solutions in $\cW_+$ exist globally and solutions in $\cW_-$ are bounded globally or collide with the origin in finite time. Moreover, for $\a\geq2$, solutions in $\cW_-$ all collide with the origin in finite time.
\end{theorem}

It is interesting to notice that for the \emph{strong} potential where $\a\geq 2$, the system exhibits significantly different behavior from the weak potential case, as can be readily seen in the Kepler problem, where $\a=2$ is a bifurcation critical value (cf. \cite{DeIb20}). Theorem \ref{thm:below} suggests that there are simple smooth boundaries that distinguish colliding orbits from global ones for $\a\geq2$ under some energy threshold, while for $\a<2$ there are no simple boundaries and indeed they seem to be fractal as suggested by the numerical results in the paper \cite{DeIb19_2}. Actually, it is well-known that the Newtonian Hill's problem has traversal homoclinic intersections, hence chaotic (cf. \cite{WaBuWi05}\cite{xia92a}). In the current paper, we will focus on the strong potential case.

\begin{theorem}[At the ground state energy]
\label{thm:at}For $\alpha\geq2$, let \begin{equation}
\begin{split}
\tilde{\cW}_+&=\{\Gamma=(x, y, \dot{x}, \dot{y}) | E(\Gamma)=E^*, W(\Gamma) > 0\}\\
\tilde{\cW}_-&=\{\Gamma=(x, y, \dot{x}, \dot{y}) | E(\Gamma)=E^*, W(\Gamma) \leq 0\}
\end{split}
\end{equation} then they are invariant. Solutions in $\tilde{\cW}_-$ either have a finite time collision or approach the ground states as $t\to\infty$, moreover, they approach the ground states only when they are on the stable manifolds of the ground states. Solutions in $\tilde{\cW}_+$ exist globally.
\end{theorem}

When the energy is above the ground state energy, both bottle necks open up, and we no longer have invariant sets based on the sign of $W$. But we can still describe the motion of the solutions when they are near the ground states and the bottle necks. 

\begin{theorem}[Above the ground state energy]
\label{thm:above}
For $\alpha\geq 2$, there exists $\epsilon>0$, so that any solution $\psi(t)$ with $E(\psi)<E^*+\epsilon^2$ either stays inside a $2\epsilon$ ball of the ground states, or ejected out of the ball. Moreover, if at the time of the exit of the ball, the sign of $W$ is negative, then the solution collides with the origin in finite time.
\end{theorem}

We remark that the finite time collision conclusion in Theorem \ref{thm:above} does not hold for the weak potential. In particular, for the Newtonian case, there are heteroclinic intersections between the two ground states, and solutions may exit the ball with $W<0$ and then come back to the ball. For instance, see figure 7 of \cite{WaBuWi05}. 

On the other hand, if at the time of exit the sign of $W$ is positive, we conjecture that the solution will escape to infinity. This is somehow supported by the results in the Newtonian case. In \cite{LMaSi85}, the authors showed that the unstable manifolds for the ground states in the outer region (i.e. $W>0$) goes forward to infinity for the Newtonian case. Their proof relies on the increasing of the argument along the orbits for a suitable complex coordinate system, which is the main theorem of \cite{Mc69}. However, for the strong potential, it seems that we don't have the result of \cite{Mc69}. We will explain more about this at the end of section \ref{sec:above}. Nonetheless, we did some numerical computations to give evidence of the no-return property (hence the escape) for the positive case for $\alpha\geq2$, see the companion paper \cite{DeIb19_2}.  In this paper, we will leave the ``no-return'' property of the positive case (i.e. $W>0$ at the exit) as a conjecture. 

Also note that in the current paper, we did not describe the dynamics of the global solutions. This part is in a subsequent work in preparation \cite{DeIbNa19}. The current paper is more on the characterization of the global existence and singularity of the solutions.

The methods of the paper are motivated from PDE, in particular, nonlinear dispersive equations (e.g. Klein-Gordon, NLS) as studied by Nakanishi-Schlag \cite{NaSc11}\cite{NaSc12} and many others. Using this method, the authors have also given some conditional characterizations about the global existence and singularity of the general N-body problem in \cite{DeIb20}. It is important to emphasize that Nakanishi-Schlag \cite{NaSc11} extended the results of Theorem \ref{thm:above} to the case where $W$ is positive, using the infinite-dimensional aspect of the Klein-Gordon equation. Indeed, the striking difference between the finite-dimensional and infinite-dimensional models shows up. For example, small initial data leads to global solutions for the Klein-Gordon equation (cf. Lemma 2.2 \cite{NaSc11}), which is not true in the Hill's problem, as clearly seen in Theorem \ref{thm:below} for small initial data they collide.

The paper is organized as follows. In section \ref{sec:ground}, we give the definition of the ground state energy. In section \ref{sec:below}, \ref{sec:at}, \ref{sec:above}, we give the proofs of the above Theorems. In the last section (appendix) we give a derivation of the Hill's type lunar problem from the general three-body problem.  In the paper, we shall use HLP to refer to the Hill's lunar problem (\ref{eq:hlp}).

\section{Defining the ground state}\label{sec:ground}
In PDE, the ground states are the solitons with the lowest energy. In the N-body problem, the solutions that play a similar role to the solitons are the relative equilibria, i.e. equilibrium under a uniform rotating frame. For the HLP, the only relative equilibria are given by the two Lagrange points, thus we shall expect them to be the ground states. In particular, we consider the following variational problem in $\R^4$: \begin{equation}\inf\{E(x,y,\dot{x},\dot{y})| W=0\}, \end{equation} where $W=-xV_x-yV_y=(\a+2)(x^2-\frac{\alpha}{r^\alpha})$.

For reasons why we consider the above type of variational problem, we refer the readers to \cite{DeIb20}.

\begin{lemma}
\label{lem:ground}
Let $E^*:=\inf\{E(x,y,\dot{x},\dot{y})| W=0\}$, then $E^*$ is finite and it's achieved exactly by the relative equilibira $\pm Q=(\pm\alpha^{\frac{1}{\alpha+2}},0,0,0)$.
\end{lemma}
\begin{proof}
Since $E=\frac{1}{2}(\dot{x}^2+\dot{y}^2)+V(x,y)$, its minimum with the constraint $W=0$ must satisfy $\dot{x}=\dot{y}=0$. Next it's easy to see the minimum of $V(x,y)$ with $W(x,y)=0$ cannot be $-\infty$, thus we use the Lagrange multiplier $\nabla V=\lm \nabla W$, the only solution of which is $\lm=0$, i.e. the minimum is achieved at the critical points of $V$. 
\end{proof}

\begin{definition}[Ground state]
We call $E^*$ the ground state energy of the HLP, and the relative equilibria $\pm Q$ the ground states. Moreover, we have $$E^*=E(\pm Q)=-\frac{1}{2}(\a+2)^2\a^{-\frac{\a}{\a+2}}.$$
\end{definition}

Now, in order to study the fate of the solutions, we investigate the behavior of the distance between $(x,y)$ and the singular point $(0,0)$. In particular, let $I:=\frac{1}{2}(x^2+y^2)$ be the \emph{moment of inertia}. If $(x(t), y(t))$ is a solution to (\ref{eq:hlp}), then \begin{equation} \frac{d^2I}{dt^2} =\dot{x}^2+\dot{y}^2+2(x\dot{y}-\dot{x}y)-xV_x-yV_y.  \end{equation}
Let \begin{equation}\label{eq:Kdef}K(x,y,\dot{x},\dot{y}):=\dot{x}^2+\dot{y}^2+2(x\dot{y}-\dot{x}y)-xV_x-yV_y,\end{equation} then the sign of $K$ will describe the behavior of $I$ along a solution, hence the fate. In the rest of this section, we give some variational properties of $K$ in terms of the ground state energy $E^*$, which will play an important role in the proof of our main theorems. In particular, we prove the following proposition:

\begin{proposition}\label{prop:KWvar}
For $\a\geq 2$,\[\inf\{E | K\geq0, W\leq 0\}=\inf\{E | K=0, W=0\}=E^*.\]
\end{proposition}

The rest of this section is dedicated to the proof of Proposition \ref{prop:KWvar}. First we can show that for $\alpha\geq2$, if $E$ is bounded and $|K|$ is small, then $(x, y)$ must be away from the singular point $(0,0)$, see Lemma \ref{lem:away_sing} and Lemma  \ref{lem:away_sing2}.

 \begin{lemma}\label{lem:away_sing}
For $\alpha>2$, there exists $c>0$ such that if $E(\Gamma)<1$ and $|K(\Gamma)|<c$, then $r\geq c$.
 \end{lemma}
 \begin{proof}
Suppose not, then there exists a sequence $\Gamma_n\in\R^4$ such that $E(\Gamma_n)<1$, $|K(\Gamma_n)|<\frac{1}{n}$ and $r_n< \frac{1}{n}$. Notice that 
 \begin{equation}
 E(\Gamma)=\frac{1}{2}(\dot{x}^2+\dot{y}^2)-\frac{\a+2}{2}x^2-\frac{\a+2}{r^\alpha},
 \end{equation}  thus $|K(\Gamma_n)|<\frac{1}{n}$ and $r_n< \frac{1}{n}$ imply that $\dot{x}_n^2+\dot{y}_n^2\sim (\a+2)\alpha n^{\alpha}$, thus $E(\Gamma_n)\sim (\a+2)(\frac{\alpha}{2}-1)n^\alpha\to\infty$, contradiction.   \end{proof}

Notice that the bound $E<1$ is not essential, in fact, the lemma can be extended to $E<a$ for any positive constant $a$. For our purpose of the paper, $a$ doesn't have to be positive. The energy we are interested in will only be slightly larger than that of $E^*$, when $\a=2$, $E^*=-4\sqrt{2}$, thus we can state a similar lemma for $\a=2$ as

\begin{lemma}\label{lem:away_sing2}
For $\alpha=2$, there exists $c>0$ such that if $E(\Gamma)<-1$ and $|K(\Gamma)|<c$, then $r\geq c$. 
 \end{lemma}
 \begin{proof}
Suppose not, then there exists a sequence $\Gamma_n\in\R^4$ such that $E(\Gamma_n)<-1$, $|K(\Gamma_n)|<\frac{1}{n}$ and $r_n< \frac{1}{n}$. Notice that 
 \begin{equation}
 E(\Gamma)=\frac{1}{2}(\dot{x}^2+\dot{y}^2)-2x^2-\frac{4}{r^2},
 \end{equation} and \begin{equation}
 K(\Gamma)=(\dot{x}-y)^2+(x+\dot{y})^2+3x^2-y^2-\frac{8}{r^2},
 \end{equation} thus $|K(\Gamma_n)|<\frac{1}{n}$ and $r_n< \frac{1}{n}$ imply that \[|\dot{x}_n^2+\dot{y}_n^2-\frac{8}{r_n^2}|\lesssim\frac{1}{n},\] thus $E(\Gamma_n)\to 0$, contradiction.   \end{proof}

\begin{remark}\label{rmk:away_sing}Similarly, we can show that for $\a\geq2$, there exists $c>0$, if $E(\Gamma)<-1$ and $K(\Gamma)\geq 0$, then $r\geq c$.\end{remark}


\subsection{Lagrange equation with one constraint}\label{sec:Lag1}
In this subsection, let's consider the Lagrange multiplier equation: $\nabla E=\lambda \nabla K$ with $K=0$. In particular, we have \begin{equation}\label{eq:Lag}\begin{cases}&V_x=\lambda(2\dot{y}-V_x-xV_{xx}-yV_{yx}), \qquad (a)\\&V_y=\lambda(-2\dot{x}-V_y-xV_{xy}-yV_{yy}),\qquad (b)\\&\dot{x}=\lambda(2\dot{x}-2y),\qquad(c)\\&\dot{y}=\lambda(2\dot{y}+2x). \qquad (d)\end{cases}\end{equation}

First, for $\lambda=0$, the equilibria of (\ref{eq:hlp}) are solutions to (\ref{eq:Lag}). We will denote $\Gamma_0:=\pm Q$ as either of the two equilibria since they are symmetric, notice that the corresponding $r_0^{\alpha+2}=\alpha$. Next, it's easy to see that $\lambda\neq\frac{1}{2}$, we then get \[\dot{x}=\frac{-2\lambda y}{1-2\lambda}, \quad \dot{y}=\frac{2\lambda x}{1-2\lambda}.\] Plug into (\ref{eq:Lag})(a)(b) and $K=0$ we will get 

\begin{equation}
\label{eq:Lag2}
\begin{split}
&x[\frac{(\a+2)-4(\a+1)\lambda^2}{1-2\lambda}+\frac{\a(\a+2)(\a\lambda-1)}{r^{\alpha+2}}]=0,\qquad (a)\\
&y[\frac{4\lambda^2}{1-2\lambda}+\frac{\a(\a+2)(\a\lambda-1)}{r^{\alpha+2}}]=0, \qquad (b)\\
&\frac{4\lambda(1-\lambda)(x^2+y^2)}{(1-2\lambda)^2}+(\a+2)(x^2-\frac{\alpha}{r^\alpha})=0.\qquad (c)
\end{split}
\end{equation}

\textbf{Case I:} When $xy\neq 0$, from (\ref{eq:Lag2})(a)(b) we get $\lambda=-\frac{1}{2}$. Thus we get four solutions for (\ref{eq:Lag}), we will denote any one of the four solutions as $\Gamma_1$ since they are symmetric with $r_1^{\alpha+2}=\a(\a+2)^2$, and \begin{equation} \quad x_1^2=\frac{\alpha}{r_1^\alpha}+\frac{3r_1^2}{4(\a+2)},\quad y_1^2=-\frac{\alpha}{r_1^\alpha}+\frac{(4\a+5)r_1^2}{4(\a+2)}, \quad \dot{x}_1=\frac{y_1}{2}, \quad \dot{y}_1=-\frac{x_1}{2}.\end{equation}

\begin{lemma}
The energy $E(\Gamma_0)<E(\Gamma_1)<0$.
\end{lemma}

\begin{proof}
We have \begin{equation}E(\Gamma_0)=-\frac{\a+2}{2}r_0^2-\frac{\a+2}{r_0^\alpha}=-\frac{(\alpha+2)^2}{2r_0^\alpha},\end{equation} and \begin{equation}E(\Gamma_1)=\frac{1}{8}r_1^2-\frac{\a+2}{2}(\frac{\alpha}{r_1^\alpha}+\frac{3r_1^2}{4(\a+2)})-\frac{\a+2}{r_1^\alpha}=-\frac{(\alpha+2)^3}{4r_1^\alpha}.\end{equation}Notice that we have used $r_0^{\alpha+2}=\alpha$ and $r_1^{\alpha+2}=\a(\a+2)^2$. Both energies are negative. Let \[f(\alpha):=\frac{E(\Gamma_1)}{E(\Gamma_0)}=\frac{1}{2}(\a+2)^{\frac{2-\a}{\a+2}},\] we have $f'(\alpha)=\frac{f(\alpha)}{(\alpha+2)^2}[2-\a-4\ln(\alpha+2)]<0$ for all $\alpha>0$. Since $f(0)=1$, we get $f(\alpha)<1$, hence $E(\Gamma_0)<E(\Gamma_1)$.
\end{proof}

\textbf{Case II:} When $x=0$, then $r^2=y^2$, from (\ref{eq:Lag2})(b)(c) we will get \begin{equation}\begin{cases}&\frac{4\lambda^2}{1-2\lambda}=\frac{\a(\a+2)(1-\a\lambda)}{r^{\alpha+2}},\\&\frac{4\lambda(1-\lambda)}{(1-2\lambda)^2}=\frac{\alpha(\a+2)}{r^{\alpha+2}}.\end{cases}\end{equation} Thus \[\lambda(1-2\lambda)=(1-\alpha\lambda)(1-\lambda),\]  \[\Rightarrow (2+\alpha)\lambda^2-(2+\alpha)\lambda+1=0,\] whose solutions are $\lambda_{\pm}=\frac{1}{2}(1\pm\sqrt{\frac{\alpha-2}{\alpha+2}})$ when $\alpha>2$. Thus we get four solutions to (\ref{eq:Lag}), two for each of the $\lambda_{\pm}$. We will denote any one of the four solutions as $\Gamma_2$ (and $\Gamma_{2\pm}$ for $\lambda_{\pm}$). For both of the $\lambda_{\pm}$, we have $r_2^{\alpha+2}=\frac{\alpha(\a+2)(\alpha-2)}{4}$. Using the relation $\dot{x}=\frac{-2\lambda y}{1-2\lambda}$, we can show that

\begin{lemma}
For $\alpha>2$, the energy $E(\Gamma_{2-})<0<E(\Gamma_{2+})$ and $E(\Gamma_1)<E(\Gamma_{2-})$.
\end{lemma}
\begin{proof}
Notice that $\Gamma_2$ has the form $(0, y, \dot{x}, 0)$, we have \begin{equation}E(\Gamma_2)=\frac{1}{2}\cdot\frac{4\lambda^2r_2^2}{(1-2\lambda)^2}-\frac{\a+2}{r_2^{\alpha}}.\end{equation} Plug in $\lambda=\frac{1}{2}(1\pm\sqrt{\frac{\alpha-2}{\alpha+2}})$ and $r_2^{\alpha+2}=\frac{\alpha(\a+2)(\alpha-2)}{4}$, we get \begin{equation}E(\Gamma_{2\pm})=\frac{(\alpha+2)^2}{8r_2^\alpha}[\alpha(1\pm\sqrt{\frac{\alpha-2}{\alpha+2}})^2-\frac{8}{\alpha+2}],\end{equation} thus we have $E(\Gamma_{2-})<0<E(\Gamma_{2+})$. To compare $E(\Gamma_{2-})$ and $E(\Gamma_1)$, let 

\begin{equation}\begin{split}g(\alpha)&:=\frac{E(\Gamma_{2-})}{E(\Gamma_1)}=\frac{r_1^\alpha}{2r_2^\alpha(\alpha+2)}[\frac{8}{\alpha+2}-\alpha(1-\sqrt{\frac{\alpha-2}{\alpha+2}})^2]\\&=\frac{1}{2(\alpha+2)}\cdot(\frac{4(\alpha+2)}{\alpha-2})^{\frac{\alpha}{\alpha+2}}[\frac{8}{\alpha+2}-\alpha(1-\sqrt{\frac{\alpha-2}{\alpha+2}})^2].\end{split}\end{equation} Through \emph{Mathematica}, we know $g(\alpha)$ is deceasing for $\alpha>2$, $\lim\limits_{\alpha\to2^+}g(\alpha)=1$ and $\lim\limits_{\alpha\to\infty}g(\alpha)=0$. Thus $0<g(\alpha)<1$ for all $\alpha>2$ and $E(\Gamma_1)<E(\Gamma_{2-})<0$.\end{proof}

\textbf{Case III:} When $y=0$, then $r^2=x^2$, from (\ref{eq:Lag2})(a)(c) we will get \begin{equation}\begin{cases}&\frac{(\a+2)-4(\a+1)\lambda^2}{1-2\lambda}=\frac{\a(\a+2)(1-\a\lambda)}{r^{\alpha+2}},\\&\frac{4(\a+1)(\lambda^2-\lambda)+\a+2}{(1-2\lambda)^2}=\frac{\alpha(\a+2)}{r^{\alpha+2}}.\end{cases}\end{equation} Thus \[[(\a+2)-4(\a+1)\lambda^2](1-2\lambda)=(1-\a\lambda)[4(\a+1)(\lambda^2-\lambda)+\a+2],\]  \[\Rightarrow \lambda[4(\a+1)(\a+2)(\lambda^2-\lambda)+\a(\a+4)]=0,\] the discriminant of the quadratic equation of $\lambda$ is $16(\a+1)(\a+2)(2-\a)$. Therefore the only solution is $\lambda=0$ when $\alpha>2$ and they coincide with the equilibria $\Gamma_0$.

\textbf{Summary:} the solutions of (\ref{eq:Lag}) with $K=0$ are as follows: 

\begin{enumerate}
\item[1.] $\lambda_0=0$, this corresponds to the equilibria of the HLP, and the configuration points are on the $x$-axis. The radius $r_0^{\alpha+2}=\alpha$. 
\item[2.] $\lambda_1=-\frac{1}{2}$, there are four configuration points, they are not on the axes. The radius $r_1^{\alpha+2}=\a(\a+2)^2$.
\item[3.] $\lambda_{2\pm}=\frac{1}{2}(1\pm\sqrt{\frac{\alpha-2}{\alpha+2}})$, there are two configuration points, they are on the $y$-axis. The radius $r_2^{\alpha+2}=\frac{\alpha(\a+2)(\alpha-2)}{4}$. 
\item[4.] The values of $E$ for these solutions satisfy \[E(\Gamma_0)<E(\Gamma_1)<E(\Gamma_{2-})<0<E(\Gamma_{2+}).\] 
\end{enumerate}


\subsection{Lagrange equation with two constraints}\label{sec:Lag2}

In this subsection, let's consider the Lagrange multiplier equation: $\nabla E=\lambda \nabla K+\mu\nabla W$ with $K=0, W=0$. In particular, we have

 \begin{equation}\label{eq:Lagtwo}\begin{cases}&V_x=\lambda(2\dot{y}-V_x-xV_{xx}-yV_{yx})+\mu(V_x+xV_{xx}+yV_{yx}) \qquad (a)\\&V_y=\lambda(-2\dot{x}-V_y-xV_{xy}-yV_{yy})+\mu(V_y+xV_{xy}+yV_{yy})\qquad (b)\\&\dot{x}=\lambda(2\dot{x}-2y)\qquad(c)\\&\dot{y}=\lambda(2\dot{y}+2x). \qquad (d)\end{cases}\end{equation} First notice that $\lambda\neq\frac{1}{2}$, then we have \[\dot{x}=\frac{-2\lambda y}{1-2\lambda}, \quad \dot{y}=\frac{2\lambda x}{1-2\lambda}.\] Plug into $K=0$, \[\frac{4\lambda(1-\lambda)(x^2+y^2)}{(1-2\lambda)^2}+(\a+2)(x^2-\frac{\alpha}{r^\alpha})=0, \] together with $W=0$ we get $\lambda=0, 1$. Simplify (\ref{eq:Lagtwo}) (a)(b) we get 

\begin{equation}
\label{eq:Lagtwo2}
\begin{cases}
&x[2(\a+2)(\lambda-\mu)+(\a+2)+\frac{4\lambda^2}{1-2\lambda}+\frac{\a(\a+2)(\alpha(\lambda-\mu)-1)}{r^{\alpha+2}}]=0\qquad (a)\\
&y[\frac{4\lambda^2}{1-2\lambda}+\frac{\a(\a+2)(\alpha(\lambda-\mu)-1)}{r^{\alpha+2}}]=0. \qquad (b)
\end{cases}
\end{equation}
If $xy\neq0$, from (\ref{eq:Lagtwo2}) we will get $\mu-\lambda=\frac{1}{2}$. But when $\lambda=0, 1$, (\ref{eq:Lagtwo2})(b) will never be satisfied unless $y=0$. Therefore we must have $y=0$ (since $x\neq0$ from the constraint $W=0$). Thus the only solution for the two-constraint Lagrange equation is $\Gamma_0$, i.e. $\pm Q$.


\subsection{Proof of Proposition \ref{prop:KWvar}}
First, $K(\pm Q)=0$, thus $\inf\{E | K=0, W=0\}=E^*$ is obvious from Lemma \ref{lem:ground}.  Let's study $\inf\{E | K\geq0, W\leq 0\}$, first notice that the infimum of $E$ must be achieved, see Lemma \ref{lem:away_sing} and Lemma \ref{lem:away_sing2} and the remark after them. Let's study the Lagrange multiplier equations with inequality constraints. Introduce new variables $p, q$, then $K\geq0$ and $W\leq 0$ are equivalent to $\tilde{K}=K-p^2=0$ and $\tilde{W}=W+q^2=0$. Then the Lagrange equation $\nabla E=\lambda \nabla \tilde{K}+\mu \nabla \tilde{W}$ coincide with (\ref{eq:Lagtwo}) with two extra equations: $0=2\lambda p, 0=2\mu q$.

\textbf{Case I:} when $pq\neq0$, then $\lambda=\mu=0$ the the only possible solution for the Lagrange equation is $\Gamma_0$, but $W=-q^2\neq0$, thus no solutions.

\textbf{Case II:} when $p=0, q=0$, then $K=0, W=0$, the solution is $\Gamma_0$ as been studied in section \ref{sec:Lag2}.

\textbf{Case III:} when $p=0, q\neq0$, then $\mu=0$ and $K=0$, the equations coincide with (\ref{eq:Lag}), and we have shown that the infimum can only be attained by $\Gamma_0$ in section \ref{sec:Lag1}.

\textbf{Case IV:} when $p\neq0, q=0$, then $\lambda=0$. We get $\dot{x}=\dot{y}=0$, and $\nabla V=\mu\nabla W$, together with $W=0$, it is easy to show that the only possible solution is $\Gamma_0$, but $K=p^2\neq0$, thus no solutions.
Therefore \[\inf\{E | K\geq0, W\leq 0\}=E(\Gamma_0)=E^*.\]



\section{below the ground state energy}
\label{sec:below}

Consider four sets in the phase space with energy below $E^*$:

\begin{equation}\label{eq:fourset}
\begin{split}
\cW_+^+&=\{\cW | K(\Gamma)\geq 0, W(\Gamma) > 0\},\\
\cW_+^-&=\{\cW | K(\Gamma)< 0, W(\Gamma) > 0\},\\
\cW_-^+&=\{\cW | K(\Gamma)\geq 0, W(\Gamma) \leq 0\},\\
\cW_-^-&=\{\cW | K(\Gamma)< 0, W(\Gamma) \leq 0\},\\\end{split}
\end{equation}
where $\cW=\{\Gamma=(x, y, \dot{x}, \dot{y})\in \R^4 | E(\Gamma)<E^*\}$. Notice that the upper right sign of $\cW$ corresponds to $K$, and the lower right sign corresponds to $W$.  We have $\cW_{+}=\cW_{+}^+\cup\cW_{+}^-$ and  $\cW_{-}=\cW_{-}^+\cup\cW_{-}^-$.

\begin{figure}[h]
\begin{subfigure}[b]{0.45\textwidth}
\centering
  \includegraphics[width=\textwidth]{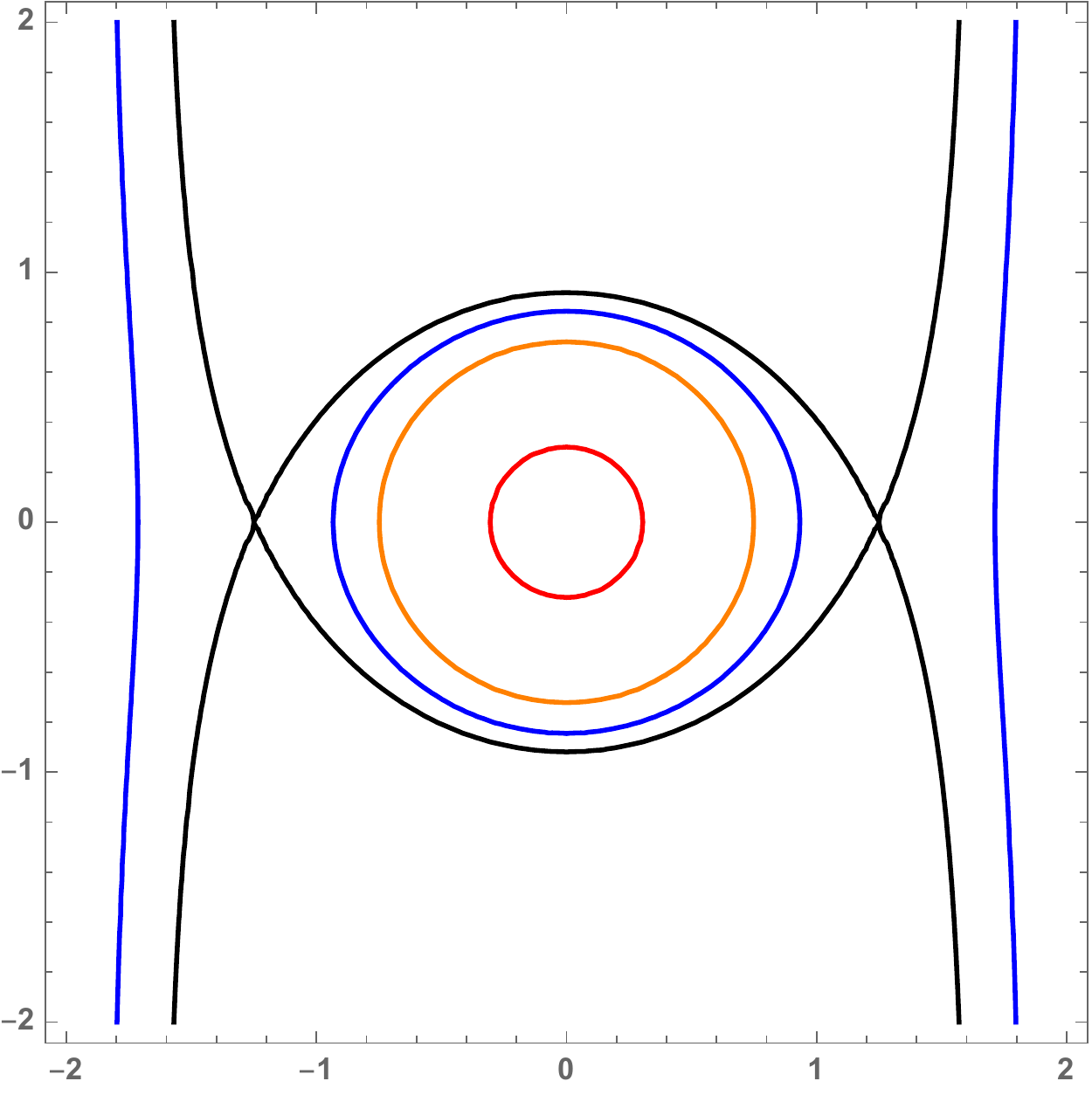}
  \caption{Level curves of $V(x, y)\leq E^*$}
  \end{subfigure}
\qquad\begin{subfigure}[b]{0.45\textwidth}
\centering
  \includegraphics[width=\textwidth]{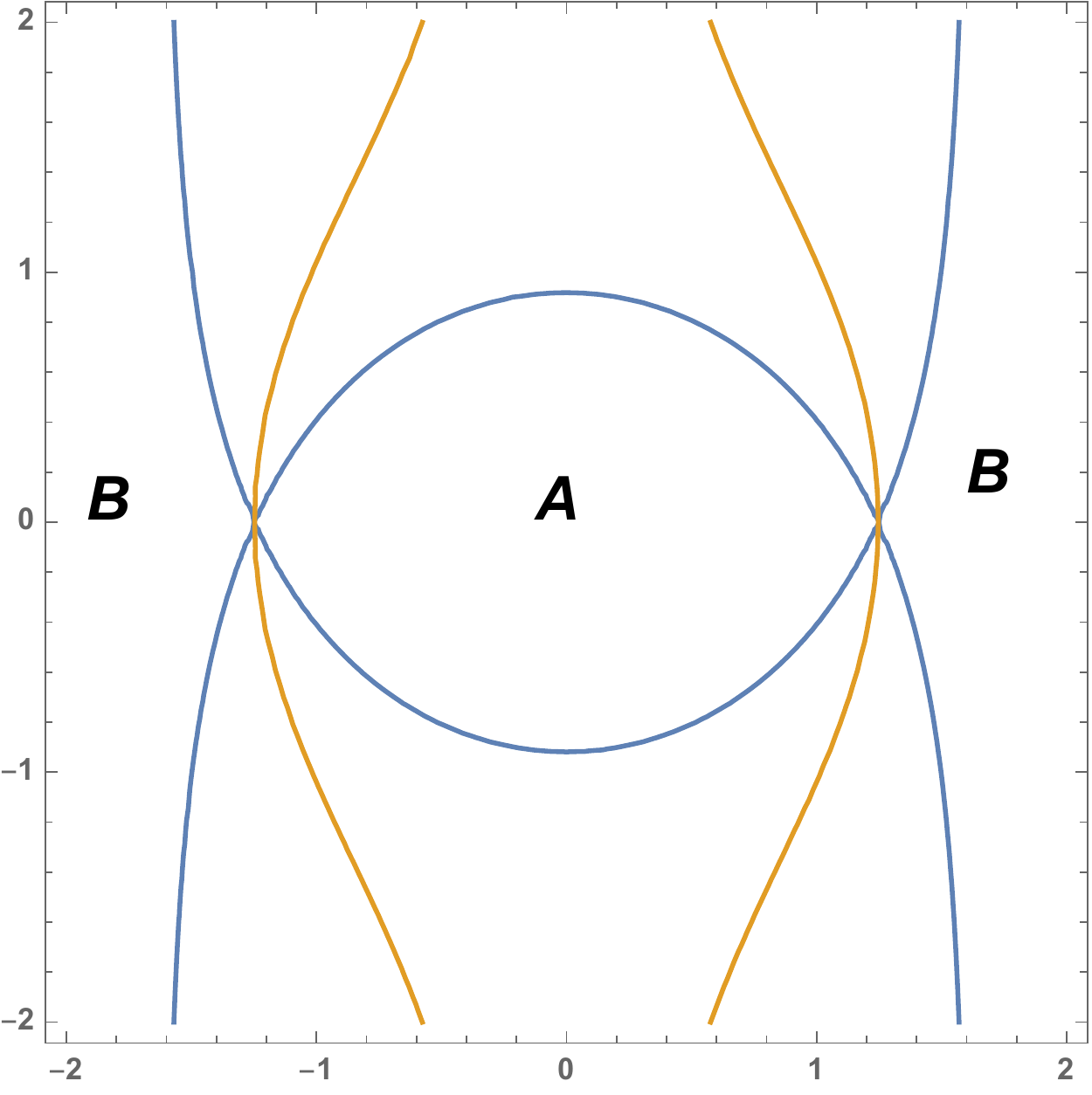}
  \caption{$V=E^*$ (blue) and $W=0$ (orange)}
  \end{subfigure}

\caption{We have $\alpha=3$ for both graphs. (A) indicates the level curves of $V(x,y)$ below $E^*$, and $V$ decreases to $-\infty$ when approaching the origin.  (B) gives the level curves of $V(x, y)=E^*$ (blue) and $W(x,y)=0$ (orange): the bounded region A in (B) is $\{V<E^*, W\leq 0\}=\mathrm{Proj} (\cW_-)$ and the two unbounded regions B in (B) is $\{V<E^*, W>0\}=\mathrm{Proj} (\cW_+)$, where $\mathrm{Proj} (\cW_{\pm})$ denotes the projection of $\cW_{\pm}$ into the configuration space $(x,y)$.}

 \label{fig:hlpVW}
\end{figure}

\begin{lemma}\label{lem:invK12}
The sets $\cW_+$ and $\cW_-$ are invariant for HLP.
\end{lemma}
\begin{proof}
It is enough to show $\cW_+=\{\Gamma\in\R^4 | E(\Gamma)<E^*, W(\Gamma)>0\}$ is invariant. Notice that the energy is conserved, so $E<E^*$ implies $V(x,y)<E^*$ in the configuration space $(x,y)$. Therefore the invariance of $\cW_+$ can be seen from the comparison of the level curves of $V(x,y)=-(\a+2)(\frac{x^2}{2}+\frac{1}{r^{\alpha}})<E^*$ and $W(x,y)=(\a+2)(x^2-\frac{\alpha}{r^{\alpha}})=0$ as illustrated in Figure \ref{fig:hlpVW}. In particular, the bounded region A in Figure \ref{fig:hlpVW}(B) corresponds to the projection of $\cW_-$ in the $(x,y)$ configuration space, and the unbounded regions B in Figure \ref{fig:hlpVW}(B) corresponds to the projection of $\cW_+$. 
\end{proof}

From Lemma \ref{lem:invK12}, we know immediately that solutions in $\cW_+$ exist globally, since $(x,y)$ stays away from the origin. Now we focus on the set $\cW_-$.

\begin{lemma}\label{lem:K1plus}
For $\alpha\geq 2$, the set $\cW_-^+$ is empty.
\end{lemma}
\begin{proof}
Use Proposition \ref{prop:KWvar}.
\end{proof}

Now we are ready to show that solutions with initial conditions in $\cW_-^-$ must have collisions.

\begin{lemma}\label{lem:kappadelta}
For $\alpha\geq 2$, given any $\delta>0$, there exists $\kappa(\delta)>0$ such that if $E(\Gamma)\leq E^*-\delta$, $W(\Gamma)\leq0$ then $|K(\Gamma)|\geq \kappa(\delta)$.
\end{lemma}

\begin{proof}
Suppose there is $\delta_0>0$, and a sequence $\{\Gamma_n\}$ satisfying $E(\Gamma_n)\leq E^*-\delta_0$ and $W(\Gamma_n)\leq0$, but $|K(\Gamma_n)|\leq \frac{1}{n}$.  By Lemma \ref{lem:away_sing} and Lemma \ref{lem:away_sing2}, there exists $M>0$, such that when $n\geq M$, the assumptions $|K(\Gamma_n)|<\frac{1}{n}$ and $E(\Gamma_n)\leq E^*-\delta_0$ imply that $r_n\geq \frac{1}{M}.$ Together with $W(\Gamma_n)\leq0$, then there exists a regular point $\Gamma_\infty$ such that $\lim\limits_{n\to\infty}\Gamma_n=\Gamma_\infty$. Therefore $K(\Gamma_\infty)=0, W(\Gamma_\infty)\leq0$ and $E(\Gamma_\infty)\leq E^*-\delta_0$ lead to a contradiction. \end{proof}

\begin{proposition}\label{prop:K1ms}
For $\alpha\geq2$, solutions in $\cW_-^-$ are singular, in particular, finite time collision.  
\end{proposition}
\begin{proof}
Let $\Gamma(t)$ be a solution in $\cW_-^-$ and $\delta=E^*-E(\Gamma)>0$, let $I(t):=I(\Gamma(t))$ and $K(t):=K(\Gamma(t))$ then by Lemma \ref{lem:kappadelta}
\[\frac{d^2 I(t)}{dt^2}=K(t)\leq -\kappa(\delta)<0.\]
Thus $I(t)\leq -\frac{1}{2}\kappa(\delta)t^2+\dot{I}(0)t+I(0)$, since $I=\frac{1}{2}(x^2+y^2)$ is always nonnegative, this implies the maximal time of existence of $\Gamma(t)$ must be finite, i.e. there is a finite time collision.
\end{proof}

Notice that $\cW_-^-=\cW_-$, thus the fate of solutions for the Hill's lunar problem with strong potential $(\alpha\geq2)$ below the ground state energy is entirely determined by the sign of $W$. In particular, solutions in the bounded region $(W\leq 0)$ are singular, and solutions in the unbounded region $(W>0)$ are global, cf. Figure \ref{fig:hlpVW}. For $\alpha<2$, the sign of $K$ in $\cW_-$ is not sign-definite, and in fact there are both collision and global solution in $\cW_-$. See the numerical investigation about this fact in our companion paper \cite{DeIb19_2}. Thus we have completed the proof of Theorem \ref{thm:below}.


\section{At the ground state energy threshold}\label{sec:at}
In this section, we study the case when $E=E^*$. We use the similar definition for the four sets as in (\ref{eq:fourset}) but now with $E(\Gamma)=E^*$. We will denote \[\tilde{\cW}:=\{\Gamma=(x, y, \dot{x}, \dot{y})\in \R^4 | E(\Gamma)=E^*\}.\] From Lemma \ref{lem:K1plus} we know  $\tilde{\cW}_-^+=\{\pm Q\}$, i.e. the equilibria (ground states) of the system. We have $\tilde{\cW}_-^-=\{ \tilde{\cW}| K<0, W\leq0\}$, then

\begin{lemma}
For $\alpha\geq 2$, the set $\tilde{\cW}_-^-$ is invariant, thus the set $\tilde{\cW}_+$ is also invariant.
\end{lemma}
\begin{proof}
When $E=E^*$ we have \[\tilde{\cW}_-=\tilde{\cW}_-^+\cup\tilde{\cW}_-^-=\{E=E^*, W\leq 0\},\] we can show $\tilde{\cW}_-$ is invariant though figure \ref{fig:hlpVW} as well as the invariance of $\tilde{\cW}_+=\{E=E^*, W> 0\}$. Furthermore, since $\tilde{\cW}_-^+=\{\pm Q\}$ is invariant, so is $\tilde{\cW}_-^-$. 
\end{proof}

\begin{proposition}
For $\alpha\geq2$, solutions in $\tilde{\cW}_-^-$ either have a finite time collision or approach the ground state as $t\to\infty$.
\end{proposition}
\begin{proof}
Let $\Gamma(t)$ be a solution in $\tilde{\cW}_-^-$, if $K(t)$ is uniformly bounded below zero, i.e. there exists $\delta>0$ so that $K(t)\leq -\delta$, then from the proof of Proposition \ref{prop:K1ms}, we know $\Gamma(t)$ experiences a finite time collision. Now if $K(t)$ is not uniformly bounded below zero, that is, there is a sequence of time $\{t_n\}$ so that $K(t_n)\to0^-$ as $t_n\to\sigma$, where $\sigma$ is the maximal time of existence. Since $\ddot{I}(t)=K(t)<0$, $\dot{I}(t)$ is decreasing, thus $\lim_{t\to\sigma}\dot{I}(t)$ exists and cannot be $-\infty$ due to $\lim_{n\to\infty}K(t_n)=0$. Assume $\lim_{t\to\sigma}\dot{I}(t)=a$, if $a<0$, then $I(t)$ exists for finite time, i.e. there is a collision. On the other hand, if $a\geq0$, then $\dot{I}(t)>0$ for all time. That is, $I(t)$ is always increasing, thus $\sigma=\infty$. Since $I(t)$ is always bounded in $\tilde{\cW}_-^-$, $a$ must be zero. Thus $\lim_{t\to\infty}K(t)=0$ and the solution $\Gamma(t)$ approaches the set $\{E=E^*, K=0, W\leq0\}=\{\pm Q\}$ as $t\to\infty$. 

\end{proof}

Moreover, we know the ground states as equilibria are of saddle-center type, see section \ref{sec:above} and Figure \ref{fig:hlpLinear}. The solution in $\tilde{\cW}_-^-$ which approaches $\pm Q$ as $t\to\infty$ only when it lies on the stable manifolds of $\pm Q$.

\begin{proposition}
Solutions in $\tilde{\cW}_+=\{E=E^*, W>0\}$ exists for all time.
\end{proposition}
\begin{proof}
Because every solution is uniformly away from the origin.
\end{proof}

Thus we have completed the proof of Theorem \ref{thm:at}.


\section{Above the ground state energy}\label{sec:above}

Following Nakanishi-Schlag \cite{NaSc11}\cite{NaSc12}, to study the fate above the ground state energy with $E<E^*+\epsilon^2$ we need to consider two scenarios: first, the \emph{stable} behavior of solutions near the ground states on the center-stable manifold of $\pm Q$; second, the solutions that enter, but then again leave, a neighborhood of the ground states. These are the \emph{non-trapped} trajectories. We will only describe the dynamics of non-trapped trajectories after they exit a ball of size $2\epsilon$, where the size is given under a suitable distance function relative to the ground states. We treat the region inside the $2\epsilon$ ball around the ground states as a ``black box'' and we do not analyze at all, since we are either trapped or being ejected out of that ball, and then carried to much larger distances from the ground states.

For the non-trapped trajectories, we study the Hill's Lunar Problem concerning the properties:

1. ``ejection lemma'', that is, solutions that do not remain close to the ground state for all positive times are ejected from any small neighborhood of it after some positive time. Moreover, the solution cannot return to that neighborhood, which is known as the ``one-pass theorem'' or ``no-return'' property.  

2. The variational estimates of the ground state energy and $K, W$. 

In the PDE case, Nakanishi-Schlag \cite{NaSc11} proved the existence of the center-stable (unstable etc.) invariant manifolds of the ground state in the infinity dimensional case. For the Hill's Lunar problem, the phase space is of dimension four, and we know $\pm Q$ are of center-saddle type. In particular, the matrix associated with the linearized Hill's equation at each ground state have a pair of real eigenvalues of equal magnitude and opposite sign \EQ{\pm k:=\pm\frac{1}{\sqrt{2}}\sqrt{\sqrt{36+36\a+29\a^2+10\a^3+\a^4}+(\a^2+3\a-2)},} and a pair of pure imaginary complex conjugate eigenvalues \EQ{\pm i\omega:=\pm \frac{i}{\sqrt{2}}\sqrt{\sqrt{36+36\a+29\a^2+10\a^3+\a^4}-(\a^2+3\a-2)}.} By standard invariant manifolds theory (cf. \cite{HPS77}) in dynamical systems, the stable, unstable, center-stable, center-unstable manifolds exist for $\pm Q$. In Figure \ref{fig:hlpLinear}, we give the tangent line of the projection of the stable and unstable manifold into the configuration space at the Lagrange point $L_2$. The qualitative picture of the flow near the equilibrium resemble the linearized flow as we shall show in the following. In order to take full advantage of the Hamiltonian nature of the HLP, we will use symplectic coordinates in the following.

\begin{figure}[ht]
	\centering
	\includegraphics[width=0.55\textwidth]{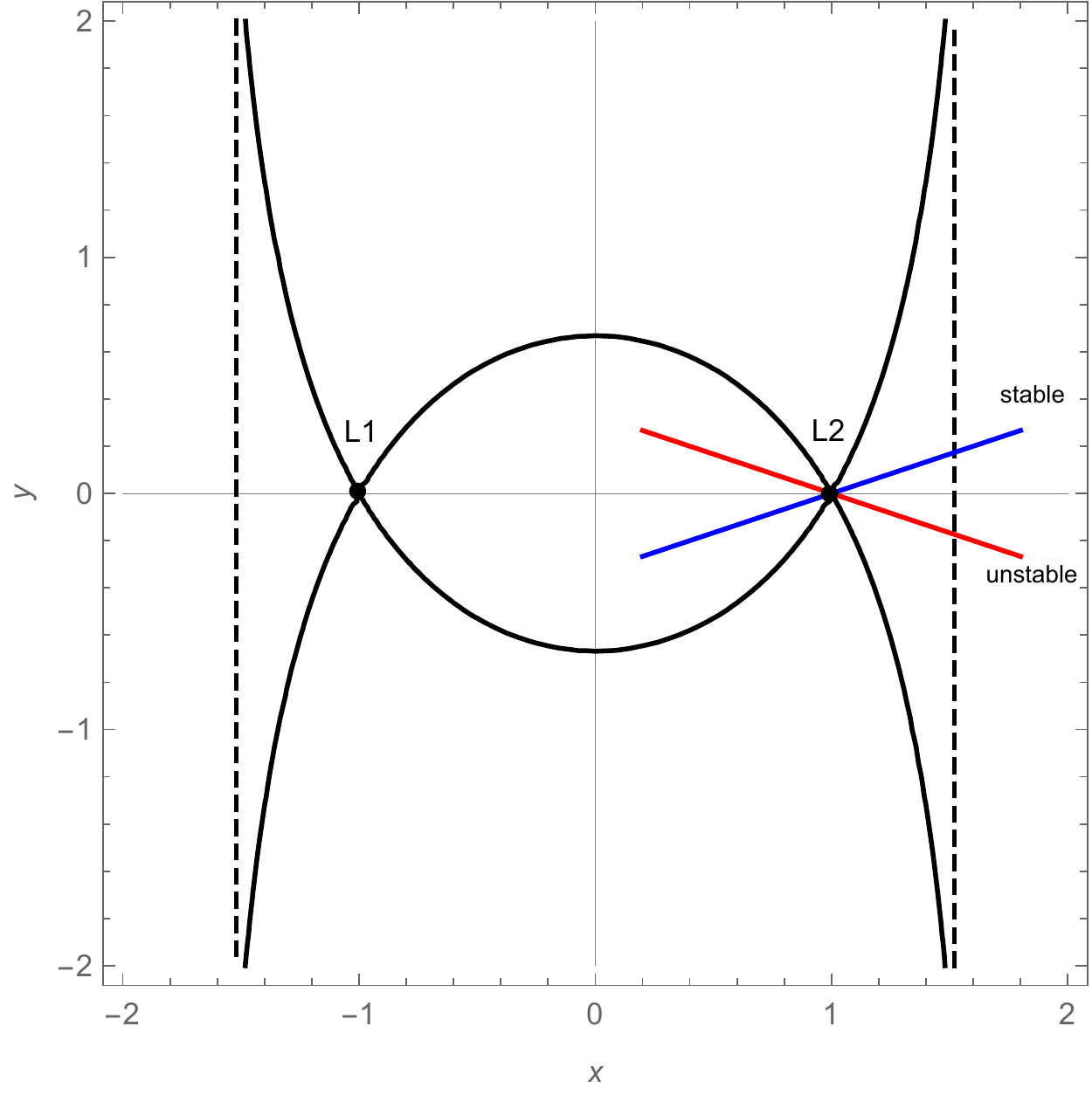}
	\caption{Zero velocity curve (black solid) of the ground state energy, and tangent lines (blue and red) of the projection of the stable and unstable manifold into the configuration space at the Lagrange point $L_2$.}
\label{fig:hlpLinear}
\end{figure}

\subsection{Symplectic algebra and Hill's equations in Symplectic form}\label{sec:sym_algebra}
Let $J=\begin{pmatrix}0&I\\-I&0\end{pmatrix}$ be the standard symplectic matrix in $\R^{2n}$ and $\Omega(\xi, \eta)=J\xi\cdot\eta=\xi^TJ^T\eta$ be the standard symplectic form on $\R^{2n}$. The symplectic group, i.e. the linear transformations that preserves $\Omega$, is \[Sp(2n)=\{A\in\R^{2n\times 2n}| A^TJA=J\}.\] The infinitesimally symplectic group is \[sp(2n)=\{A\in\R^{2n\times 2n}| A^TJ+JA=0\}.\] Note that if $A\in sp(2n)$ then $e^A\in Sp(2n)$. For a linear Hamiltonian equation $\dot{X}=JSX$ where $S^T=S$, the matrix $A=JS\in sp(2n)$ and $e^{tJS}\in Sp(2n)$ (cf. \cite{AbMa08}).

\begin{lemma}
If $A\in sp(2n)$, then $\Omega(A\xi,\eta)=-\Omega(\xi, A\eta)$. Thus if $\xi, \eta$ are eigenvectors of $A$ with eigenvalues $\lambda, \mu$, then $\Omega(\xi, \eta)=0$ provided $\lambda+\mu\neq0$.
\end{lemma}
\begin{proof}
\[\Omega(A\xi,\eta)=\xi^TA^TJ^T\eta=\xi^TJA\eta=-\xi^TJ^TA\eta=-\Omega(\xi,A\eta).\]
Then second statement follows directly from this property.
\end{proof}

Consider the Hill's equations with coordinates $q=(x,y)$ and $p=(p_x,p_y)=(\dot{x}-y, \dot{y}+x)$, the Hamiltonian, i.e. the energy is \EQ{E(x, y, p_x,p_y)=\frac{1}{2}[(p_x+y)^2+(p_y-x)^2]+V(x,y),} and the function $K$ (cf. (\ref{eq:Kdef})) is 
\EQ{\label{eq:Ksym}K(x,y,p_x,p_y)=p_x^2 +p_y^2+(\a+1)x^2-y^2-\a(\a+2)r^{-\a}.}

The Hill's equations (\ref{eq:hlp}) in Symplectic canonical form is \begin{equation}\label{eq:hlpsym}\dot{q}=\frac{\partial E}{\partial p}, \quad \dot{p}=-\frac{\partial E}{\partial q}.\end{equation} That is, \[\begin{pmatrix}\dot{q}\\\dot{p}\end{pmatrix}=J\nabla E.\]

The ground states in the $(q,p)$ coordinates are 
\EQ{\label{eq:deftQ}\pm\tilde{Q}:=\pm(q_0,0,0,q_0),\quad q_0=\a^{\frac{1}{\a+2}}.}

 Linearize (\ref{eq:hlpsym}) at $\tilde{Q}$ one obtains \[\begin{pmatrix}\dot{q}\\\dot{p}\end{pmatrix}=J\nabla^2 E(\tilde{Q})\begin{pmatrix}q\\p\end{pmatrix},\] where \begin{equation}\label{eq:hesE}\nabla^2 E(\tilde{Q})=\begin{pmatrix}1-(\a+2)^2&0&0&-1\\0&1+(\a+2)&1&0\\0&1&1&0\\-1&0&0&1\end{pmatrix}.\end{equation} Note that it is easy to check \EQ{\label{eq:hesEpmQ} \nabla^2 E(\tilde{Q})=\nabla^2 E(-\tilde{Q}),} which is extremely helpful in our analysis. 

 Let $A=J\nabla^2 E(\tilde{Q})$, the eigenvalues of $A$ are $\pm k, \pm i\omega$. Let $\xi_+$ and $\xi_-$ be the eigenvectors of $k$ and $-k$ respectively, and we assume $\xi_+$ and $\xi_-$ are normalized in the following sense: 
\begin{equation}\label{eq:defxipm}\Omega(\xi_+, \xi_-)=1,\quad \Omega(\xi_-,\xi_+)=-1,\quad |\xi_+|=|\xi_-|,\end{equation} where $|\cdot|$ denotes the Euclidean norm of a vector. It is obvious that \begin{equation}\Omega(\xi_+, \xi_+)=0,\quad \Omega(\xi_-,\xi_-)=0.\end{equation} See Figure \ref{fig:hlpLinear} for the projection of $\xi_+, \xi_-$ to the configuration space at $\tQ$.

Write $(q,p)=\tilde{Q}+X$, where $X$ is the perturbation and we have \begin{equation}\dot{X}=J\nabla^2 E(\tilde{Q})X+N(X)=AX+N(X),\end{equation} where $N(X)$ stands for the non-linear terms. Decompose $\R^4$ as $E^u\bigoplus E^s\bigoplus E^c$ where $E^{u,s,c}$ are the eigenspace for $k, -k, \pm i\omega$ respectively. Write $X$ as follows: \begin{equation}\label{eq:symdec}X=\lambda_+(t)\xi_++\lambda_-(t)\xi_-+\gamma(t),\end{equation} where \begin{equation}\gamma(t)\in E^c,\quad\Omega(\gamma(t),\xi_+)=\Omega(\gamma(t),\xi_-)=0.\end{equation} One has $\lambda_{\pm}=\pm\Omega(X, \xi_{\mp})$ and we can derive the differential equations for $\lambda_{\pm}(t)$.

\begin{equation}
\begin{split}
\frac{d\lambda_{\pm}}{dt}(t)&=\pm\Omega(\dot{X}, \xi_{\mp})\\&=\pm\Omega(AX+N(X), \xi_{\mp})\\&=\pm\Omega(AX, \xi_{\mp})\pm\Omega(N(X), \xi_{\mp})\\&=\pm k\lambda_{\pm}(t)\pm\Omega(N(X), \xi_{\mp}).
\end{split}
\end{equation} Thus

\EQ{\label{eq:lmpm}\frac{d\lambda_{+}}{dt}(t)=k\lambda_{+}(t)+N_+(X), \quad\frac{d\lambda_{-}}{dt}(t)=-k\lambda_{-}(t)-N_-(X),} where

\EQ{N_+(X)=\Omega(N(X), \xi_{-}),\quad N_-(X)=\Omega(N(X), \xi_+).} 

Moreover, we have
\begin{equation}
\label{eq:labound1}
\lambda^2_{\pm}=\Omega(X, \xi_{\mp})^2\leq|\xi_+|^2|X|^2.
\end{equation}

\subsection{Linearized energy norm}
In this section we want to define the ``linearized energy norm'' for the perturbation $X$. First, as a result of the symplectic orthogonality of $\xi_\pm$ and $\gamma$, we have the following relationship:

\begin{lemma}
\label{lem:gasim}
The function $\gamma(t)$ in the decomposition (\ref{eq:symdec}) satisfies \[\Omega(\gamma, A\gamma)\sim |\gamma|^2.\]
\end{lemma}

\begin{proof}
Let $\eta=\eta_1+i\eta_2$ be the eigenvector of $i\omega$, thus $E^c=\mathrm{Span}\{\eta_1, \eta_2\}$ and we have $A\eta_1=-\omega \eta_2$, $A\eta_2=\omega \eta_1$. One can check that $\Omega(\eta_1,\eta_2)>0$, thus we shall assume $\eta_1, \eta_2$ are normalized in the sense that \[\Omega(\eta_1, \eta_2)=1.\] Thus $\xi_+, \xi_-, \eta_1, \eta_2$ form a symplectic basis. In particular, the matrix $P=(\xi_+\,\, \eta_1\,\, \xi_-\,\, \eta_2)\in Sp(4)$ and $A=PBP^{-1}$ where \[B=\begin{pmatrix}k&0&0&0\\0&0&0&\omega\\0&0&-k&0\\0&-\omega&0&0\end{pmatrix}.\] Since $\gamma\in E^c$, we assume $\gamma=a\eta_1+b\eta_2$, thus $P^{-1}\gamma=(0,a,0,b)$.
\begin{equation}
\begin{split}
\Omega(\gamma, A\gamma)&=\Omega(\gamma, PBP^{-1}\gamma)\\
&=\Omega(P^{-1}\gamma, BP^{-1}\gamma)\\
&=\omega(a^2+b^2).
\end{split}
\end{equation}
The second equality is due to that $P$ is symplectic. The last equality is by direct computations. Since $\eta_1, \eta_2,\omega$ are constant quantities, we conclude that $\Omega(\gamma, A\gamma)\sim |\gamma|^2$. 
\end{proof}

By Taylor expansion and analysis in section \ref{sec:sym_algebra}, we have 

\begin{equation}
\label{eq:LEN0}
\begin{split}
E(\tilde{Q}+X)-E(\tilde{Q})&=\frac{1}{2}X^T\nabla^2E(\tilde{Q})X+C(X)\\&=\frac{1}{2}X^TJ^TJ\nabla^2E(\tilde{Q})X+C(X)\\&=\frac{1}{2}\Omega(X, AX)+C(X)\\&=-k\lambda_+\lambda_-+\frac{1}{2}\Omega(\gamma, A\gamma)+C(X),
\end{split}
\end{equation} note that $C(X)=O(|X|^3)$. We define the \emph{linearized energy norm} $|X|_E$ by 

\begin{equation}
|X|_E^2:=\frac{k}{2}(\lambda_+^2(t)+\lambda_-^2(t))+\frac{1}{2}\Omega(\gamma, A\gamma).
\end{equation} Then we get that 

\begin{equation}
\label{eq:LEN}
E(\tilde{Q}+X)-E(\tilde{Q})+\frac{k}{2}(\lambda_+(t)+\lambda_-(t))^2-|X|_E^2=O(|X|^3).
\end{equation}

\begin{lemma}
\label{lem:LEN}
The function $\gamma(t)$ in the decomposition satisfies \[|\gamma(t)|\lesssim |X(t)|,\] moreover, \[|X(t)|\sim|X(t)|_E.\]
\end{lemma}
\begin{proof}By the triangular inequality, Lemma \ref{lem:gasim} and (\ref{eq:labound1}).\end{proof}

\subsection{Distance function from the ground state}
We introduce the distance function relative to the ground states $\pm\tilde{Q}$. Let \[\psi=\sigma(\tilde{Q}+X), \quad X=\lambda_+\xi_++\lambda_-\xi_-+\gamma,\] where $\lambda_\pm(t)$ and $\gamma(t)$ are functions appearing in the decomposition for $X(t)$ as in (\ref{eq:symdec}), and $\sigma=\pm$ with ``+'' corresponds to the ground state $\tilde{Q}$ and ``-'' corresponds to the ground state $-\tilde{Q}$, note that $E(\tilde{Q})=E(-\tilde{Q})$.

From Lemma \ref{lem:LEN} and (\ref{eq:LEN}) we see that there exists a constant $\delta_E>0$ with the following property: for any solution $\psi=\sigma(\tilde{Q}+X)$ to (\ref{eq:hlpsym}) and any time $t\in I_{max}(\psi)$ (i.e. the maximal interval of existence for $\psi$) for which $|X(t)|_E\leq4\delta_E$, 
\begin{equation}
\label{eq:DF1}
|E(\psi(t))-E(\tilde{Q})+\frac{k}{2}(\lambda_+(t)+\lambda_-(t))^2-|X(t)|_E^2|\leq \frac{|X(t)|_E^2}{10}.
\end{equation}

Let $\chi$ be a smooth function on $\R$ such that $\chi(r)=1$ for $|r|\leq 1$ and $\chi(r)=0$ for $|r|\geq2$. We define \[d_\sigma(\psi(t)):=\sqrt{|X(t)|_E^2+\chi(|X(t)|_E/2\delta_E)C(\psi(t))},\]
where \[C(\psi(t))=E(\psi(t))-E(\tilde{Q})+\frac{k}{2}(\lambda_+(t)+\lambda_-(t))^2-|X(t)|_E^2,\] as appeared in (\ref{eq:LEN0}).
Moreover, we will let \[d_Q(\psi(t)):=\min( d_{+}(\psi(t)), d_-(\psi(t)))=d_\sigma(\psi(t)), \] notice that $\sigma$ is unique as long as $d_Q(\psi)$ is less than half the distance of the two ground states.

 It will be convenient to introduce the new parameters \[\lambda_1(t):=\frac{\lambda_+(t)+\lambda_-(t)}{2}, \quad \lambda_2(t):=\frac{\lambda_+(t)-\lambda_-(t)}{2},\] and we have 

\EQ{\label{eq:lm1lm2}\frac{d\lambda_{1}}{dt}(t)=k\lambda_{2}(t)+\frac{1}{2}N_1(X), \quad\frac{d\lambda_{2}}{dt}(t)=k\lambda_{1}(t)-\frac{1}{2}N_2(X),} where

\EQ{N_1(X)=\Omega(N(X), \xi_++\xi_{-}),\quad N_2(X)=\Omega(N(X), \xi_+-\xi_{-}).} Note that \EQ{|N_1(X)|\lec |X|^2, \quad |N_2(X)|\lec |X|^2.}

Some important properties about the function $d_Q(\psi(t))$ are in the following:
\begin{lemma}
\label{lem:DF1}
Assume that there exists an interval $I$ on which \[\mathrm{sup}_{t\in I}d_Q(\psi(t))\leq\delta_E.\] Then, all of the following hold for all $t\in I$:
\begin{itemize}
\item[(i)] $\frac{1}{2}|X(t)|^2_E\leq d_Q(\psi(t))^2\leq \frac{3}{2}|X(t)|^2_E,$
\item[(ii)] $d_Q(\psi(t))^2=E(\psi(t))-E(\tilde{Q})+2k\lambda_1^2(t),$
\item[(iii)] $\frac{d}{dt}d_Q(\psi(t))^2=4k^2\lambda_1(t)\lambda_2(t)+2k\lambda_1(t)N_1(X)$.
\end{itemize}
\end{lemma}
\begin{proof}
The proof of (i) is directly from the definition of $d_Q(\psi)$ and (\ref{eq:DF1}). To prove (ii), notice that from (i) and $d_Q(\psi(t))\leq\delta_E$, we have $|X(t)|_E<2\delta_E$, thus 
\begin{equation}d_Q(\psi(t))^2=|X(t)|_E^2+C(\psi(t))=E(\psi(t))-E(\tilde{Q})+2k\lambda_1^2(t).\end{equation} To show (iii), we differentiate (ii) with respect to $t$ on both sides:
\begin{equation}
\begin{split}
\frac{d}{dt}d_Q(\psi(t))^2&=\nabla E(\psi(t))\cdot \dot{\psi}(t)+4k\lambda_1\dot{\lambda_1}\\&=\nabla E(\psi(t))\cdot J\nabla E(\psi(t))+4k\lambda_1\dot{\lambda_1}\\&=0+4k\lambda_1[k\lambda_{2}(t)+\frac{1}{2}N_1(X)]\\&=4k^2\lambda_1(t)\lambda_2(t)+2k\lambda_1(t)N_1(X).
\end{split}
\end{equation}
\end{proof}

Moreover, we can show that when the energy is at most slightly above the ground state energy, and the solution is near the ground state, then the distance function is dominated by $|\lambda_1|$.

\begin{lemma}[Eigenmode dominance]
\label{lem:ED}
For any $\psi\in \R^4$ such that \EQ{\label{eq:ED}E(\psi)<E^*+\frac{1}{2}d_Q(\psi)^2,\quad d_Q(\psi)\leq\delta_E,} one has $d_Q(\psi)\sim |\lambda_1|$. Moreover, $\lm_1$ has a fixed sign in each connected component of the region (\ref{eq:ED}).
\end{lemma}
\begin{proof}
We have \begin{equation}\begin{split}|X|_E^2&=E(\psi)-E^*+\frac{k}{2}(\lm_++\lm_-)^2-C(\psi)\\&\leq E(\psi)-E^*+2k\lm_1^2+\frac{|X|_E^2}{10},\end{split}\end{equation} thus \[\frac{9}{10}|X|_E^2\leq E(\psi)-E^*+2k\lm_1^2.\] Using Lemma \ref{lem:DF1} (i), we obtain 
\begin{equation}d_Q(\psi)^2\leq\frac{3}{2}|X|_E^2\leq\frac{5}{6}d_Q(\psi)^2+\frac{10}{3}k\lm_1^2,\end{equation} hence \begin{equation}d_Q(\psi)^2\leq 20k\lm_1^2.\end{equation}

On the other hand, \begin{equation}\lm_1^2\leq\frac{1}{2}(\lm_+^2+\lm_-^2)\leq|\xi_+|^2|X|^2\lesssim d_Q(\psi)^2,\end{equation} where the last inequality is from Lemma \ref{lem:LEN} and Lemma \ref{lem:DF1} (i). Finally, inside the set (\ref{eq:ED}) one can never have $\lm_1=0$, since that would mean both $d_Q(\psi)=0$ and $E(\psi)<E^*$ which is impossible. 
\end{proof}

\subsection{The ejection process}

\begin{lemma}[Ejection Lemma]
\label{lem:ejlem}
There exists constants $0<\delta_X\leq\delta_E$ and $C_*, T_*$ with the property: If $\psi(t)$ is a local solution to (\ref{eq:hlpsym}) on $[0, T]$ satisfying \begin{equation}\label{eq:ejlem1}R:=d_Q(\psi(0))\leq\delta_X, \quad E(\psi)<E^*+\frac{1}{2}R^2,\end{equation} then we can extend $\psi(t)$ as long as $d_Q(\psi(t))\leq\delta_X$. Furthermore, if there exists some $t_0\in (0, T)$ such that \begin{equation}\label{eq:ejcon}d_Q(\psi(t))\geq R,\quad \forall 0<t<t_0,\end{equation} and let \[T_X:=\inf\{t\in[0,t_0]: d_Q(\psi(t))=\delta_X\}\] where $T_X=t_0$ if $d_Q(\psi(t))<\delta_X$ on $[0,t_0]$, then for all $t\in[0,T_X]:$ 
\begin{itemize}
\item[(i)] $d_Q(\psi(t))\sim \mathfrak{s}\lm_1(t)\sim \mathfrak{s}\lm_+(t)\sim e^{kt}R$,
\item[(ii)] $|\lm_-(t)|+|\gamma(t)|\lesssim R+R^{\frac{3}{2}}$,
\item[(iii)]$\mathfrak{s}K(\psi(t))\gec d_Q(\psi(t))-C_*d_Q(\psi(0))$, and $\mathfrak{s}W(\psi(t))\gec d_Q(\psi(t))-C_*d_Q(\psi(0))$.
\end{itemize}
where $\mathfrak{s}=+1$ or $-1$. Moreover, $d_Q(\psi(t))$ is increasing for $t\in [T_*R, T_X]$ and $|d_Q(\psi(t))-R|\lec R^3$ for $t\in [0, T_*R]$.
\end{lemma}
\begin{proof}
First notice that when the solution is close to $\pm\tQ$ means it is away from the singular point, thus the solution extends as long as $d_Q(\psi(t))$ is small. Next, from the definition of $T_X$ and (\ref{eq:ejcon}) we have for any $t\in[0, T_X]$, \EQ{R\leq d_Q(\psi(t))\leq \delta_X<\delta_E,} thus from Lemma \ref{lem:DF1} and Lemma \ref{lem:ED} for any $t\in[0, T_X]$,  
\EQ{\label{eq:ej1}|\gamma(t)|\lec |X(t)|\sim |X(t)|_E\sim d_Q(\psi(t))\sim |\lm_1(t)|,}
\EQ{\label{eq:ej2}\frac{d}{dt}d_Q(\psi(t))^2=4k^2\lambda_1(t)\lambda_2(t)+2k\lambda_1(t)N_1(X).}
Hence for any $t\in[0, T_X]$,
\EQ{0<R\lec |\lm_1(t)|,} which together with the continuity of $\lm_1(t)$ shows that for any $t\in[0, T_X]$ \EQ{\sk:=\mr{sign}[\lm_1(t)]=\mr{sign}[\lm_1(0)].} The ejection condition (\ref{eq:ejcon}) implies $\frac{d}{dt}d_Q(\psi(t))^2\vert_{t=0}\geq0$. Since $|N_1(X)|\lec|X|^2\lec\lm_1^2$, we deduce from (\ref{eq:ej2}) that $\sk\lm_2(0)\gec-|\lm_1(0)|^2$ and so $\lm_+(0)\sim\lm_1(0)\sim\sk R$.

Integrating the equation for $\lm_{\pm}$ yields 
\EQ{|\lm_{\pm}(t)-e^{\pm kt}\lm_{\pm}(0)|\lec \int_0^t e^{k(t-s)}|N_{\pm}(X(s))|ds \lec \int_0^t e^{k(t-s)}|\lm_1(s)|^2ds,} thus by continuity in time we deduce that as long as $Re^{kt}\ll 1$, we have 
\EQ{\label{eq:ej3}\lm_1(t)\sim \lm_+(t)\sim\sk Re^{kt}, \quad |\lm_{\pm}(t)-e^{\pm kt}\lm_{\pm}(0)|\lec R^2e^{2kt}.} Thus completes the proof of (i) and that $|\lm_-(t)|\lec R+R^2e^{2kt}$. Now to estimate $|\gamma(t)|$, we introduce the nonlinear energy projected to $\{\xi_+, \xi_-\}$ plane (cf. (\ref{eq:LEN0})): 
\EQ{E_{\{\xi_+,\xi_-\}}(t):=E(\tQ+\lm_+\xi_++\lm_-\xi_-)-E(\tQ)=-k\lm_+\lm_-+C(\lm_+\xi_++\lm_-\xi_-).} Denote $\xi(t):=\lm_+(t)\xi_++\lm_-(t)\xi_-$, we get 
\EQ{|\frac{d}{dt}E_{\{\xi_+,\xi_-\}}(t)|&=|\nabla E(\tQ+\xi)\cdot \dot{\xi}|\\&=|(\nabla^2E(\tQ)\xi+O(|\xi|^2))\cdot\dot{\xi}|\\&\lec |\nabla^2E(\tQ)\xi\cdot\dot{\xi}|+|\xi|^2|\dot{\xi}|\\&=|\Omega(\dot{\xi},A\xi)|+|\xi|^2|\dot{\xi}|\\&=k|\dot{\lm}_+\lm_-+\lm_+\dot{\lm}_-|+|\xi|^2|\dot{\xi}|\\&=k|\lm_-N_+(X)+\lm_+N_-(X)|+|\xi|^2|\dot{\xi}|\\&\lec (|\lm_-|+|\lm_+|)^3+(|\lm_-|+|\lm_+|)^2(|\dot{\lm}_+|+|\dot{\lm}_-|)\\&\lec(|\lm_-|+|\lm_+|)^3+(|\lm_-|+|\lm_+|)^2|\gamma(t)|^2. } Hence 
\EQ{\label{eq:ej4}|\frac{d}{dt}(E(\psi(t))-E_{\{\xi_+,\xi_-\}}(t))|\lec (|\lm_-|+|\lm_+|)^3+(|\lm_-|+|\lm_+|)^2|\gamma(t)|^2.} Moreover,
\EQ{\label{eq:ej5}E(\psi(t))-E(\tQ)-E_{\{\xi_+,\xi_-\}}(t)&=\frac{1}{2}\Omega(\gamma, A\gamma)+O(|X|^3)+O(|\xi|^3)\\&\sim |\gamma|^2+O(|X|^3)+O(|\xi|^3),} thus by  (\ref{eq:ej5}) (\ref{eq:ej4}) and (\ref{eq:ej3}) we get the desired estimate in (ii) for $|\gamma|$. The equation of $\lm_2$ implies $\sk\lm_2(t)\gec R(e^{kt}-1)-R^2$, hence there is $T_*\sim 1$ such that $\sk\lm_2\gec R$ and $\frac{d}{dt}d_Q(\psi)>0$ for $t\geq T_*R$. For $t\in [0,T_*R]$, we have $|\frac{d}{dt}d_Q(\psi)|\lec R^2$ and so $|d_Q(\psi)-R|\lec R^3$.

Finally, we expand $K$ around $\tQ$: 
\EQ{K(\tQ+X)=\nabla K(\tQ)\cdot X+O(|X|^2).} Let $Q_0:=(-q_0, 0, 0, q_0)$ then from (\ref{eq:Ksym}) and (\ref{eq:hesE}) one can check that \EQ{\nabla K(\tQ) = ((\a^2+4\a+2)q_0, 0, 0, 2q_0)=\nabla^2E(\tQ)Q_0.} Thus 
\EQ{\label{eq:ej6}K(\tQ+X)&=\nabla^2E(\tQ)Q_0\cdot X+O(|X|^2)\\&=\Omega(Q_0, AX)+O(|X|^2)\\&=k\lm_+\Omega(Q_0, \xi_+)-k\lm_-\Omega(Q_0,\xi_-)+\Omega(Q_0, A\gamma)+O(|X|^2).}
We can choose the normalized eigenvector $\xi_+, \xi_-$ (cf. (\ref{eq:defxipm})) so that $\Omega(Q_0, \xi_+)>0$, then we obtain $\sk K(\psi(t))\gec (e^{kt}-C_*)R$ from (\ref{eq:ej6}) (\ref{eq:ej3}) and the estimates on $|\lm_-(t)|+|\gamma(t)|$ in (ii) as we have proved. The proof for $W$ is similar, and this time notice that 
\EQ{\label{eq:ej7}\nabla W(\tQ)=((\a+2)^2q_0, 0, 0, 0)=-\nabla^2E(\tQ)\tQ,}
and one can check that the previous chosen $\xi_+$ will satisfy $-\Omega(\tQ,\xi_+)>0$. 
\end{proof}

\begin{remark}\label{rmk:ejcon}Note that for the ejection condition (\ref{eq:ejcon}) to be satisfied, one can consider for instance $\lm_-(0)=0, \gamma(0)=0, \lm_+(0)=R_0$, then differential equation of $(\lm_+, \lm_-,\gamma)$ shows that $\lm_+(t)$ grows exponentially.\end{remark}

\begin{remark}\label{rmk:ejback}Note that the ejection Lemma also works backward in time, this time it's the stable manifold, i.e. the eigenvalue $-k$ dominates. Since $\lm_1=\frac{\lm_++\lm_-}{2}$, we get the same statements as in the lemma for backward time.
\end{remark}

\begin{remark}
The ejection lemma holds for all $\alpha>0$. However, to show the one-pass theorem we need some crucial variational estimates about the sign of $K$, and that estimate only works for $\alpha\geq 2$. 
\end{remark}


\subsection{The variational estimates}
Let $$d_L(\Gamma):=\min\{ \mathrm{dist}((x,y), L_1), \mathrm{dist}((x,y), L_2)\}$$ be the distance in the configuration space with respect to the Lagrange points. The following lemma implies that the sign of $W, K$ cannot change outside of a neighborhood of $L_1, L_2$ that is not too small; the size here depends on the amount by which the energy of the solution exceeds that of $L_1, L_2$.

\begin{figure}[h]

\centering
  \includegraphics[width=0.5\textwidth]{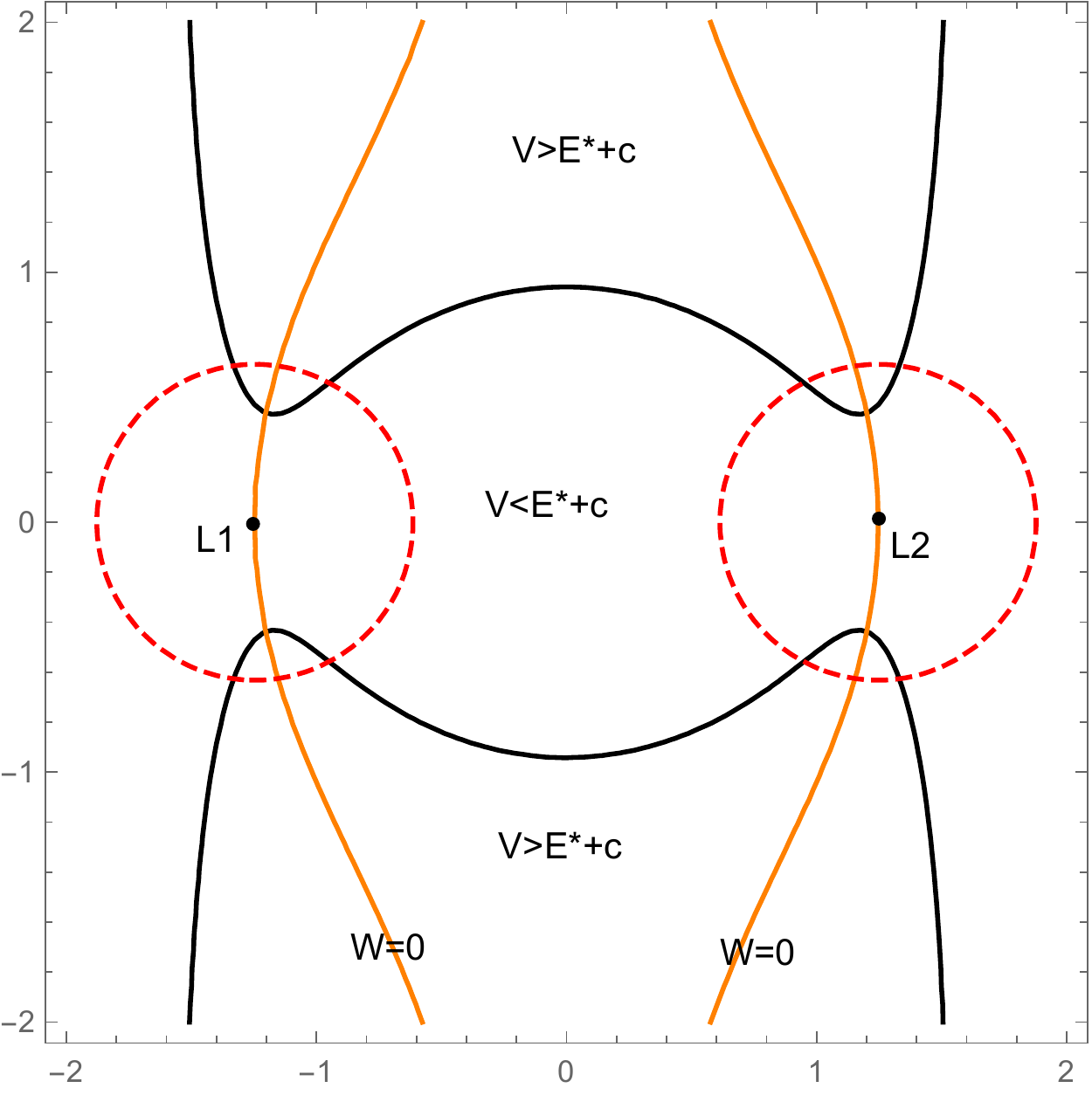}
  \caption{The black curve is the zero velocity curve for $E(\Gamma)=E^*+c$, i.e. $V(x,y)=E^*+c$. When $c=\epsilon(\delta)^2$ is small enough, the value of $|W|$ is  uniformly away from 0, provided $d_L(\Gamma)>\delta$, i.e. outside the red dashed circles.}
  \label{fig:EWabove}
\end{figure}

\begin{lemma}
\label{lem:var_phy}
For the strong force $\alpha\geq 2$, for any $\delta>0$, there exist $\epsilon(\delta), \kappa(\delta)>0$ such that for any $\Gamma\in \R^4$ satisfying \begin{equation}\label{eq:var_phy}E(\Gamma)<E^*+\epsilon(\delta)^2, \quad d_L(\Gamma)\geq\delta,\end{equation}one has either \[W(\Gamma)\leq-\kappa(\delta)\quad \textrm{and}\quad K(\Gamma)\leq-\kappa(\delta),\] or \[W(\Gamma)\geq\kappa(\delta).\]
\end{lemma}
\begin{proof}
First, notice that there exist $\epsilon(\delta), \kappa(\delta)>0$ so that if (\ref{eq:var_phy}) is satisfied then $|W(\Gamma)|\geq \kappa(\delta)$, see figure \ref{fig:EWabove}. We prove the first case by contradiction. Suppose there is $\delta_0>0$ and a sequence $\Gamma_n$ with \[d_L(\Gamma_n)\geq\delta_0,\quad E(\Gamma_n)<E^*+\frac{1}{n},\quad K(\Gamma_n)>-\frac{1}{n}\] and there exists $N_0, \kappa_0>0$ so that $W(\Gamma_n)\leq-\kappa_0$ for all $n\geq N_0$. By lemma \ref{lem:away_sing}, lemma \ref{lem:away_sing2} and remark \ref{rmk:away_sing}, we know there is a regular point $\Gamma_\infty$ so that $\Gamma_n\to\Gamma_\infty$. Thus \[ d_L(\Gamma_\infty)\geq\delta_0, \quad E(\Gamma_\infty)\leq E^*,\quad K(\Gamma_\infty)\geq0,\quad W(\Gamma_\infty)\leq-\kappa_0,\]which is impossible since $\inf\{E|W\leq0, K\geq0\}=E^*$ and is only achieved by $\pm Q$. 
\end{proof}
The variational estimate in Lemma \ref{lem:var_phy} uses the distance function in the configuration space $(x,y)$. We also have the variational estimate if we use the distance function $d_Q(\psi(t))$ in the phase space $(q,p)$.

\begin{lemma}
\label{lem:var_DF}
For the strong force $\alpha\geq 2$, for any $\delta>0$, there exist $\epsilon(\delta), \kappa(\delta)>0$ such that for any $\psi\in \R^4$ satisfying \begin{equation}\label{eq:above1}E(\psi)<E^*+\epsilon(\delta)^2, \quad d_Q(\psi)\geq\delta,\end{equation}one has either \[W(\psi)\leq-\kappa(\delta)\quad \textrm{and}\quad K(\psi)\leq-\kappa(\delta),\] or \[W(\psi)\geq\kappa(\delta).\]
\end{lemma}
\begin{proof}
Use the fact that $d_Q(\psi(t))\sim |X|=\min(|\psi(t)-\tQ|, |\psi(t)+\tQ|)$ and an obvious modification of the proof as in Lemma \ref{lem:var_phy}. 
\end{proof}

We remark that in the above lemma, when $W>0$, the sign of $K$ can be both positive and negative with the condition (\ref{eq:above1}). Nonetheless, Lemma \ref{lem:var_DF} and the Ejection Lemma \ref{lem:ejlem} still enable us to define a sign function away from $\pm\tQ$ by combining those of $\lm_1$ and $W$.

\begin{lemma}[The sign function]
\label{lem:signfun}
For the strong force $\alpha\geq2$, let $\delta_S:=\frac{\delta_X}{2C_*}>0$ where $\delta_X$ and $C_*>1$ are constants from Lemma \ref{lem:ejlem}. Let $0<\delta\leq\delta_S$ and 
\EQ{\cH_{(\delta)}:=\{\psi\in\R^4 | E(\psi)<E^*+\min(d_Q^2(\psi)/2, \epsilon(\delta)^2)\},} where $\epsilon(\delta)$ is given by Lemma \ref{lem:var_DF}. Then there exists a unique continuous function $\kS : \cH_{(\delta)}\to \{\pm1\}$ satisfying 
\EQ{\begin{cases}\psi\in\cH_{(\delta)}, d_Q(\psi)\leq\delta_E &\Longrightarrow\quad \kS(\psi)=\mathrm{sign}[\lm_1],\\\psi\in\cH_{(\delta)}, d_Q(\psi)\geq\delta &\Longrightarrow\quad \kS(\psi)=\mathrm{sign}[W(\psi)],\end{cases}}where we set $\mathrm{sign}[0]=+1$.
\end{lemma}
\begin{proof}
The proof is the same as that in Nakanishi-Schlag (\cite{NaSc11}, Lemma 4.9). In particular, Lemma \ref{lem:ED} implies that $\lm_1$ is continuous for $d_Q(\psi)\leq\delta_E$, and Lemma \ref{lem:var_DF} implies that $\mathrm{sign}[W(\psi)]$ is continuous for $d_Q(\psi)\geq\delta$. Thus it suffices to show that they coincide at $d_Q(\psi)=\delta_S\in[\delta,\delta_X]$ in $\cH_{(\delta)}$. Let $\psi(t)$ be a solution of the HLP with $\psi(0)\in\cH_{(\delta)}, d_Q(\psi(0))=\delta_S$, and such that the ejection condition (\ref{eq:ejcon}) is satisfied. The ejection condition is easy to achieve, for instance see Remark \ref{rmk:ejcon}. Then the Ejection Lemma implies that $\psi(t)$ stays in $\cH_{(\delta)}$ and that $\mathrm{sign}[\lm_1(t)]$ is constant, until $d_Q(\psi(t))$ reaches $\delta_X$, at which time $\mathrm{sign}[W(\psi(t))]$ is the same as $\mathrm{sign}[\lm_1(t)]$, because $2C_*\delta_S\leq \delta_X$. Since $\mathrm{sign}[W(\psi)]$ is constant for $d_Q(\psi)\geq\delta$, we conclude that $\mathrm{sign}[W(\psi(t))]=\mathrm{sign}[\lm_1(t)]$ from the beginning $t=0$.
\end{proof}

\subsection{The one-pass theorem}

In this section we prove the crucial no-return property of the solutions for $\alpha\geq 2$. Indeed, thanks to Lemma \ref{lem:var_DF} we are able to prove this for the case when $\kS=-1$. When $\kS=+1$, we lose the control of the sign of $K$, and the method from Nakanishi-Schlag seem to fail. We conjecture that the no-return property still holds for $\kS=+1$ and provide another direction about the proof of the no-return property for this case.  We remark that the no-return property means that there are no homoclinic or heteroclinic orbits relative to $\pm\tQ$. The following theorem implies that there are no heteroclinic orbits between the two equilibria, moreover, when a trajectory exits the $2\epsilon$ ball of $\tQ$ or $-\tQ$ with $\kS=-1$, i.e. ejecting into the region $W<0$, then it will collide with the origin. This is very different from the Newtonian Hill's Lunar problem, where an orbit enters the region $W<0$ may still leave it through one of the bottle necks. See Figure 7 in \cite{WaBuWi05}.   

\begin{theorem}[One-pass theorem]
\label{thm:onepass}
For the strong force $\alpha\geq 2$, there exist constants $0<\epsilon_*\ll R_*\ll 1$ with the property: for any solution $\psi$ of the HLP (\ref{eq:hlpsym}) on the maximal interval $[0,T_{\mathrm{max}})$ satisfying \[E(\psi)<E^*+\epsilon^2, \quad d_Q(\psi(0))<R,\] for some $\epsilon\in(0,\epsilon_*]$, $R\in(2\epsilon, R_*]$, and let \EQ{T_{\mathrm{trap}}:=\sup\{t_*\geq 0|  d_Q(\psi(t))<R+R^{2}, \forall \,\,t\in [0, t_*)\},} then one has the following:
\begin{enumerate}
\item[(i)] if $T_{\mathrm{trap}}=T_{\mathrm{max}}$, then $\psi$ is ``trapped'' at $\pm\tQ$ and $T_{\mathrm{max}}=\infty$;
\item[(ii)]  if $T_{\mathrm{trap}}<T_{\mathrm{max}}$, and $\kS(T_\mathrm{trap})=-1$, then $\kS(t)$ does not change on $[T_{\mathrm{trap}}, T_{\mathrm{max}})$, and $d_Q(\psi(t))\geq R+R^{2}$ for all $T_{\mathrm{trap}}\leq t< T_{\mathrm{max}}$. Moreover, the solution has a finite time collision.
\end{enumerate}
\end{theorem}

The proof relies on a virial type argument ($\ddot{I}=K$) and is similar to that of PDE, cf. Theorem 4.1 in \cite{NaSc12}, Theorem 4.11 in \cite{NaSc11} and Theorem 7.1 in \cite{AIKN19}. In the paper, we follow closely with that in \cite{NaSc12}. We choose $R+R^2$ for the dichotomy because the distance function may exhibit oscillations on the order of $O(R^3)$, and so we need some room to ensure a true ejection from the small neighborhood. The most work of the proof is on the no-return statement, as the finite time collision then follows easily. Indeed, the finite time collision in the $\kS=-1$ region follows from $W<-\kappa<0,$ hence $K<-\kappa<0$ and we can use the argument as in Proposition \ref{prop:K1ms}. Thus the one-pass theorem implies Theorem \ref{thm:above}.

We now turn to the proof of the no-return statement. We may assume that $\psi$ does not stay very close to $\pm\tQ$ for all $t>0$, so that we can apply the Ejection Lemma \ref{lem:ejlem} at some time $t_*= T_{\mathrm{trap}}>0$. The idea is to combine the hyperbolic structure of Lemma \ref{lem:ejlem} close to $\pm\tQ$ with the variational structure in Lemma \ref{lem:var_DF} away from $\pm\tQ$, in order to control the virial identity through $K(\psi)$. We choose $\delta_*>0$ as the distance threshold between the two regions in $\R^4$: for $d_Q(\psi)<\delta_*$ we use the hyperbolic estimate in Lemma \ref{lem:ejlem}, and for $d_Q(\psi)>\delta_*$ we use the variational estimate in Lemma \ref{lem:var_DF}. So $\delta_*, \epsilon_*,R_*$ should satisfy 
\EQ{\label{eq:OP0}\epsilon_*\ll R_*\ll\delta_*\ll \delta_S,\quad \epsilon_*\leq\epsilon(\delta_*).} We shall impose further smallness conditions on $\delta_*, R_*, \epsilon_*$. Afterward, $R_*$ and then $\epsilon_*$ need to be made even smaller in order to satisfy the above conditions, depending on $\delta_*$.

Suppose towards a contradiction that $\psi$ solves the HLP (\ref{eq:hlpsym}) on $[0,T_\mathrm{max})$ in $\R^4$ satisfying for some $0<T_1<T_2<T_3<T_\mathrm{max}$ and all $t\in(T_1, T_3)$, 
\EQ{\label{eq:OP1}d_Q(\psi(0))<R=d_Q(\psi(T_1))=d_Q(\psi(T_3))<d_Q(\psi(t)), \quad d_Q(\psi(T_2))\geq R+R^2,} as well as $E(\psi)<E^*+\e^2$ for some $\e\in(0,\e_*]$ and $R\in(2\e,R_*]$. 

Lemma \ref{lem:signfun} implies that $\sk:=\kS(\psi(t))\in\{\pm1\}$ is well-defined and constant on $T_1\leq t\leq T_3$.

We apply the Ejection Lemma \ref{lem:ejlem} first from $t=T_1$ forward in time. Then by the lemma, there exists $T_1'\in(T_1, T_1+T_*R)$ such that $d_Q(\psi(t))$ increases  for $t>T_1'$ until it reaches $\de_X$, and $d_Q(\psi(T_1'))=R+O(R^3)<d_Q(\psi(T_2))\ll\de_X$. Hence $T_1<T_1'<T_2$, and by the lemma there is  $T_1''\in(T_1', T_3)$ such that $d_Q(\psi(t))$ increases exponentially on $(T_1',T_1'')$, $d_Q(\psi(T_1''))=\de_X$ and on $(T_1, T_1'')$, 
\EQ{\label{eq:OP2}d_Q(\psi(t))\sim e^{k(t-T_1)}R,\quad \sk K(\psi(t))\gec(e^{k(t-T_1)}-C_*)R.} We can argue in the same way from $t=T_3$ backward in time to obtain a time interval $(T_3'',T_3)\subset(T_1'',T_3)$, so that $d_Q(\psi(T_3''))=\de_X$,
\EQ{\label{eq:OP3}d_Q(\psi(t))\sim e^{k(T_3-t)}R,\quad \sk K(\psi(t))\gec(e^{k(T_3-t)}-C_*)R, \quad t\in(T_3'',T_3),} and $d_Q(\psi(t))$ is decreasing in the region $d_Q(\psi(t))\geq R+R^2$. Moreover, from any $\tau\in(T_1'', T_3'')$ where $d_Q(\psi(\tau))<\de_*$ is a local minimum, we can apply the Ejection Lemma both forward and backward in time, thereby obtaining an open interval $\bI_\tau\subset (T_1'', T_3'')$ so that $d_Q(\psi(\partial \bI_\tau))=\{\de_X\}$, 
\EQ{\label{eq:OP4}d_Q(\psi(t))\sim e^{k|t-\tau|}d_Q(\psi(\tau)),\quad \sk K(\psi(t))\gec(e^{k|t-\tau|}-C_*)d_Q(\psi(\tau)),\quad t\in \bI_\tau,} and $d_Q(\psi(t))$ is monotone in the region $d_Q(\psi(t))\geq 2d_Q(\psi(\tau))$, which is the reason for $\bI_\tau\subset(T_1'',T_3'')$. Moreover, the monotonicity away from $\tau$ implies that any two intervals $\bI_{\tau_1}$ and $\bI_{\tau_2}$ for distinct local minimal points $\tau_1$ and $\tau_2$ are either disjoint or identical. Therefore, we have obtained disjoint open subintervals $\bI_1, \cdots, \bI_n\subset(T_1,T_3)$ with $n\geq 2$, where we have either (\ref{eq:OP2}), (\ref{eq:OP3}) or (\ref{eq:OP4}) with $\tau=\tau_j\in \bI_j$. 

At the remaining times 
\EQ{\label{eq:OP5} t\in \bI':=(T_1, T_3)\setminus\bigcup_{j=1}^n\bI_j,} we have $d_Q(\psi(t))\geq\de_*$, so that we can apply Lemma \ref{lem:var_DF} to obtain for $t\in\bI'$,
\EQ{\label{eq:OP6}\begin{cases}W(\psi(t))\leq -\kappa(\delta_*), K(\psi(t))\leq-\kappa(\delta_*)&(\sk=-1),\\W(\psi(t))\geq \kappa(\delta_*)&(\sk=+1).\end{cases}}

We shall call the $\bI_j$'s the ``hyperbolic'' intervals, and $\bI'$ the ``variational'' intervals. See Figure \ref{fig:OP}.

\begin{figure}[h]
	\centering
	\begin{tikzpicture}

	\draw [->] (0, 0)--(10, 0)node [below] {$t$};
	\draw [->](0, 0)--(0, 8) node  [left] {$d_Q$};
	\draw (0, 7.3)--(10, 7.3) node at (0,7.3) [left] {$\delta_X$};
  \draw (0, 5)--(10, 5) node at (0,5) [left] {$\delta_*$};
  \draw (0, 1.5)--(10, 1.5) node at (0,1.5) [left] {$R+R^2$};
  \draw (0, 1)--(10, 1) node at (0,1) [left] {$R$};
   \draw [thick] (0,0.2) to [out=75, in=180] (0.2,0.6) to [out=350,in=170] (0.4,0.3) to           [out=350, in=240] (0.6,0.5) to [out=60,in=230] (0.8,1)to [out=60,in=230] (0.87,1.3);
  \draw[thick] (0.87,1.3) to [out=0,in=170] (1,1.1);
  \draw[thick] (1,1.1) to [out=0,in=260] (1.2,1.5) to [out=80,in=260] (2.2,7.3) to [out=75,in=180] (2.5,8) to [out=0,in=100] (2.7,7) to [out=350, in=240] (2.8,7.4) to [out=75,in=100] (3,7.3);
 \draw [thick] (3.9,3.1) to [out=120,in=280] (3,7.3);  
 \draw [thick] (3.9,3.1)to [out=280,in=170] (4,3)  to [out=0,in=70] (4.1,3.1)  to [out=70,in=260] (5,7.4)  to [out=60,in=200] (5.3,7.6)  to [out=0,in=95] (5.4,7.5)  to [out=280,in=170] (5.8,6.5)  to [out=0,in=60] (6,6.7) to [out=70,in=260] (6.5,8)  to [out=80,in=95] (6.7,7.9)  to [out=300,in=100] (8,3);
     \draw [thick] (8.5,1) to [out=120, in=280] (8,3);

\draw [thick,dotted] (0.8,1) -- (0.8,0);
\draw [thick,dotted] (1,1.1) -- (1,0);
\draw [thick,dotted] (1.4,2.5) -- (1.4,0);
\draw [thick,dotted] (2.2,7.3) -- (2.2,0);
\draw [thick,dotted] (3,7.3) -- (3,0);
\draw [thick,dotted] (5,7.3) -- (5,0);
\draw [thick,dotted] (7,7.3) -- (7,0);
\draw [thick,dotted] (8.5,1) -- (8.5,0);

\draw [thick,dotted] (4,3) -- (4,0);

\node  at (0, 0) [below]  {0};
\node  at (0.7, 0) [below]  {\tiny $T_1$};		
\node at (1.1, 0) [below]  {\tiny $T_1'$};		
\node at (1.5, 0) [below]  {\tiny $T_2$};		
\node at (2.2, 0) [below]  {\tiny $T_1''$};
\node at (4, 0) [below]  { $\tau$};				
\node at (7, 0) [below]  {\tiny $T_3''$};	
\node  at (8.5, 0) [below]  {\tiny $T_3$};

\draw [fill] (0,0.2) circle [radius=0.05];
\draw [fill] (0.8,1) circle [radius=0.05];
\draw [fill] (1,1.1) circle [radius=0.05];
\draw [fill] (1.4,2.6) circle [radius=0.05];
\draw [fill] (2.2,7.3) circle [radius=0.05];
\draw [fill] (3,7.3) circle [radius=0.05];
\draw [fill] (5,7.3) circle [radius=0.05];
\draw [fill] (7,7.3) circle [radius=0.05];
\draw [fill] (4,3) circle [radius=0.05];
\draw [fill] (8.5,1) circle [radius=0.05];

\draw [ultra thick, red] (0.8,0)--(2.2,0);
\draw [ultra thick, red] (3,0)--(5,0);
\draw [ultra thick, red] (7,0)--(8.5,0);
\draw [ultra thick, blue] (2.2,0)--(3,0);
\draw [ultra thick, blue] (5,0)--(7,0);

\draw [thick] (3,-0.5)--(3,-0.8);
\draw [thick] (5,-0.5)--(5,-0.8);	
\draw [thick,->] (3,-0.65)--(3.6,-0.65);
\draw [thick,->] (5,-0.65)--(4.4,-0.65);
\node at (4,-0.65) {$\bI_\tau$};
	
	\end{tikzpicture}
	\caption{ Behavior of $d_Q(\psi(t))$ in $(T_1, T_3)$ when there is a local minimum at $\tau$ with $d_Q(\psi(\tau))<\delta_*$. The red intervals indicate the ``hyperbolic'' region and the blue intervals indicate the ``variational'' region. The curve is smooth.} \label{fig:OP}
\end{figure}
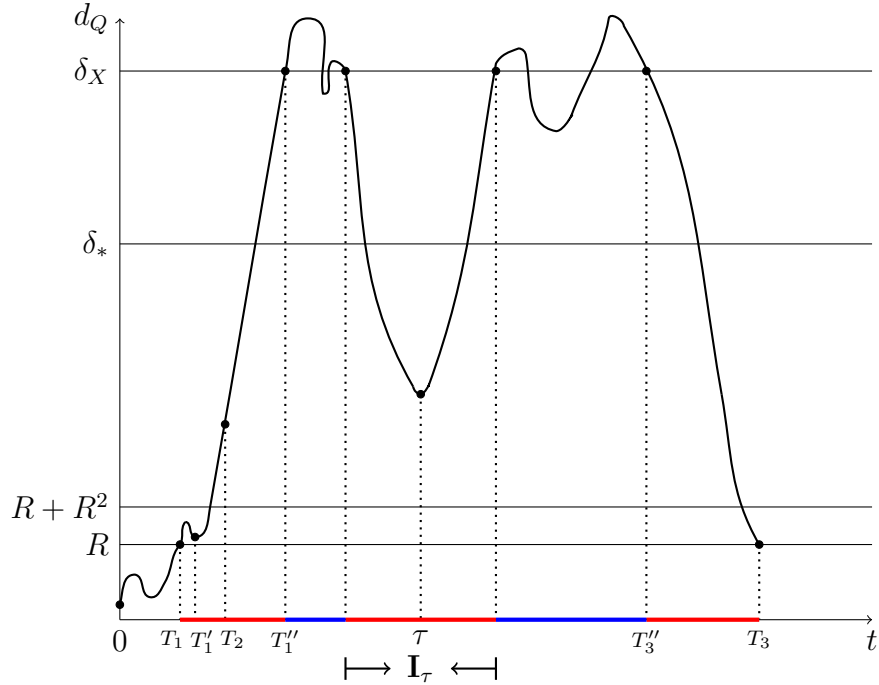

\subsubsection{Virial estimate in the finite time collision case $\sk=-1$}
Notice that the virial identity is 
\EQ{\label{eq:OP7}\frac{d^2}{dt^2}I(\psi(t))=K(\psi(t)).}Multiply (\ref{eq:OP7}) by $\sk=-1$ and integrate over $(T_1,T_3)$. Combining (\ref{eq:OP2})-(\ref{eq:OP6}) we obtain 
\EQ{\label{eq:OP8}|\dot{I}(\psi(T_3))-\dot{I}(\psi(T_1))|&\gec \sum_{j=1}^n\int_{\bI_j}(e^{k|t-\tau_j|}-C_*)d_Q(\psi(\tau_j))dt+\int_{\bI'}\kappa(\delta_*)dt\\&\gec n\delta_X\geq \delta_X.} On the other hand, 
\EQ{\dot{I}=x\dot{x}+y\dot{y}=x(p_x+y)+y(p_y-x)=xp_x+yp_y,}thus $\dot{I}(\pm\tQ)=0$. At $t=T_1, T_3$ we have $d_Q(\psi(t))=R$, therefore
\EQ{|\dot{I}(\psi(T_3))-\dot{I}(\psi(T_1))|\lec R\leq R_*\ll \de_X.} Comparing this bound with (\ref{eq:OP8}) leads to a contradiction. In conclusion, the solution $\psi(t)$ cannot return to the $R$-ball from the $\sk=-1$ side.

\subsubsection{No-return property for the case $\sk=+1$} In this case, we lose the sign definiteness of $K$ in the variational region thus the above virial type argument seems to fail for $\sk=+1$. However, we conjecture that the no-return property should still hold. This is somehow supported by the results in the Newtonian case. In \cite{LMaSi85}, the authors showed that the unstable manifolds for the equilibria in the outer region (i.e. $W>0$) goes forward to infinity for the Newtonian case. Their proof relies on the increasing of the argument $\theta$ along the orbits for a suitable complex coordinate system for the general restricted three-body problem (R3BP), which is the main theorem of \cite{Mc69}. In particular, McGehee \cite{Mc69} showed that all orbits (with energy slightly above $L_2$) of the Newtonian R3BP leaving a vicinity of one of the equilibira proceed around an invariant annulus before returning to that vicinity. The Hill's lunar problem is a limiting case of the R3BP, and that is how Llibre-Mart\'inez-Sim\'o (cf. section 4 \cite{LMaSi85}) showed the unstable manifolds in the outer region for $L_1, L_2$ must go to infinity (no-return). The idea is roughly sketched in Figure \ref{fig:McLMS}.

\begin{figure}[h]

\begin{subfigure}[b]{0.45\textwidth}
\centering
  \includegraphics[width=\textwidth]{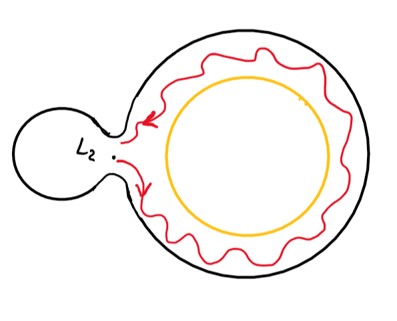}
 
  \end{subfigure}
\qquad\begin{subfigure}[b]{0.45\textwidth}
\centering
  \includegraphics[width=\textwidth]{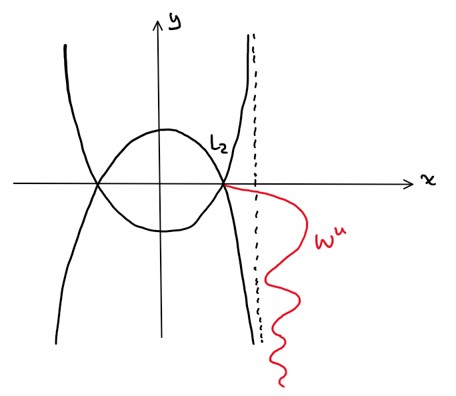}
 
  \end{subfigure}
\caption{Left: main technical result of McGehee \cite{Mc69} for the R3BP. Right: the unstable manifolds of the equilibria of the HLP in the outer region Llibre-Mart\'inez-Sim\'o (cf. section 4 \cite{LMaSi85})}

\label{fig:McLMS}
\end{figure}

Initially, we attempted to show the same result for the R3BP with $\alpha\geq2$ as in \cite{Mc69}, but it seems that this is not the case. Instead of going around  an invariant annulus (such an invariant annlus may not exist for strong potential R3BP) in the configuration space, all solutions leaving a neighborhood of the $L_2$ very likely will collide with one of the primary bodies in finite time. In particular, we conjecture that for the strong potential R3BP, all solutions with energy below the second Lagrange point (this energy is the threshold energy when one has a bounded component of the Hill's region for the R3BP) should collide with one of the primary bodies in finite time if they started in the bounded component of its Hill's region. Assuming this conjecture, then by passing to a limit similar to that of \cite{LMaSi85}, we can show that the unstable manifolds of the equilibira in the outer region $(W>0)$ of the strong potential HLP go to infinity, hence completing the no-return property for $\sk=+1$.

We remark that characterizing the fates of the solutions for the strong potential R3BP itself is a very interesting problem, and it is a more complicated problem than the HLP. It is our hope that more attention could be put to this problem for a broader community.

\section*{acknowledgement}

The authors are grateful to Takafumi Akahori and Kenji Nakanishi for helpful discussions. The first author is supported by Sun Yat-sen University start-up grant No. 74120-18841213. The second author is supported by the NSERC grant No. 371637-2019.


\section*{Appendix: Defining the Hill's lunar problem}
Following \cite{MeSc82}, we introduce scaled symplectic coordinates into the general three body problem with $\alpha$-potential so that the Hill's lunar equations are the equations of the first approximation. We remark that there is no essential difficulty in extending \cite{MeSc82} to the $\alpha$-potential, we include this section for completeness. 

Consider a uniform rotating frame with frequency equal to one with reference to a fixed inertial frame and let $x_0, x_1, x_2; y_0, y_1, y_2\in\R^2$ be the position and momentum vectors relative to the rotating frame of three particles of masses $m_0, m_1, m_2$. The Hamiltonian defining the equations of motion of the three particles with $\alpha$-potential is 

\begin{equation}\label{eq:H}H=\sum\limits_{i=0}^2(\frac{|y_i|^2}{2m_i}-x_i^TJy_i)-\sum\limits_{i<j}\frac{m_im_j}{|x_i-x_j|^\alpha},\end{equation}where $J=\begin{pmatrix}0&1\\-1&0\end{pmatrix}$. We shall refer to the particles of mass $m_0, m_1$ and $m_2$ as the earth, moon and sun, respectively. Since the center of mass is preserved and we want to scale the distances between the earth and moon, we use the Jacobi coordinates, which is a symplectic change of coordinates:

\begin{equation}\label{eq:jocobi}
\begin{split}
u_0&=(m_0+m_1+m_2)^{-1}(m_0x_0+m_1x_1+m_2x_2),\\
u_1&=x_1-x_0,\\
u_2&=x_2-(m_0+m_1)^{-1}(m_0x_0+m_1x_1),\\
v_0&=y_0+y_1+y_2,\\
v_1&=(m_0+m_1)^{-1}(m_0y_1-m_1y_0),\\
v_2&=(m_0+m_1+m_2)^{-1}[(m_0+m_1)y_2-m_2(y_0+y_1)],
\end{split}
\end{equation} and obtain

\begin{equation}\label{eq:H0}H=\sum\limits_{i=0}^2(\frac{|v_i|^2}{2M_i'}-u_i^TJv_i)-\frac{m_0m_1}{|u_1|^\alpha}-\frac{m_1m_2}{|u_2-\nu_0'u_1|^\alpha}-\frac{m_0m_2}{|u_2+\nu_1'u_1|^\alpha},\end{equation}where

\begin{equation}
\begin{split}
&M_0'=m_0+m_1+m_2, \qquad M_1'=(m_0+m_1)^{-1}m_0m_1, \\
&M_2'=(m_0+m_1+m_2)^{-1}(m_0+m_1)m_2,\\
&\nu_0'=(m_0+m_1)^{-1}m_0, \qquad \nu_1'=(m_0+m_1)^{-1}m_1.
\end{split}
\end{equation}
Since $H$ is independent of $u_0$ (the center of mass), its conjugate variable $v_0$ (total linear momentum) is an integral. As usual, we take $u_0=v_0=0$, thus we shall proceed with the Hamiltonian defined in (\ref{eq:H0}) with the summation index from $i=1$ to 2. We will make various assumptions on the size of various quantities until we are led to a definition of the equation of the first approximation for lunar theory. Following \cite{MeSc82}, we will use the common arrow notation in scaling problems. 

The first assumption is that the earth and moon have approximately the same mass and are small compared to the mass of the sun, thus we let 

\begin{equation}m_0=\epsilon^{2\gamma}\mu_0, \qquad m_1=\epsilon^{2\gamma}\mu_1, \qquad m_2=\mu_2,\end{equation}
where $\epsilon$ is a small positive parameter and $\gamma$ is a positive number to be chosen later. Since the masses of $m_0$ and $m_1$ are of order $\epsilon^{2\gamma}$, so will be their momenta provided their velocities are of order 1. We make the substitutions $v_1\to \epsilon^{2\gamma}v_1, v_2\to \epsilon^{2\gamma}v_2$ in (\ref{eq:H0}). With this symplectic change of variables with multiplier $\epsilon^{2\gamma}$ the Hamiltonian becomes

\begin{equation}\label{eq:H1}
\begin{split}
H&=H_1+H_2+O(\epsilon^{2\gamma}),\\
H_1&=\frac{|v_1|^2}{2M_1}-u_1^TJv_1-\frac{\epsilon^{2\gamma}\mu_0\mu_1}{|u_1|^\alpha},\\
H_2&=\frac{|v_2|^2}{2M_2}-u_2^TJv_2-\frac{\mu_1\mu_2}{|u_2-\nu_0u_1|^\alpha}-\frac{\mu_0\mu_2}{|u_2+\nu_1u_1|^\alpha},
\end{split}
\end{equation}
where \begin{equation}
\begin{split}
&M_1=(\mu_0+\mu_1)^{-1}\mu_0\mu_1, \qquad M_2=\mu_0+\mu_1, \\
&\nu_0=(\mu_0+\mu_1)^{-1}\mu_0, \qquad \nu_1=(\mu_0+\mu_1)^{-1}\mu_1.
\end{split}
\end{equation}
Note that the $O(\epsilon^{2\gamma})$ only depends on $|v_1|$ and $|v_2|$.

The next assumption is that the distance between the earth and moon ($|u_1|$) is small relative to the distance between the sun and the center of mass of the earth-moon system ($|u_2|$). We make the change of variables $u_1\to \epsilon^{2\delta}u_1$, where $\delta$ is a positive number to be chosen later. Following \cite{MeSc82}, we use the Legendre polynomials to expand the $\alpha$-potential terms in $H_2$. In particular, if $u, u'$ are two vectors with $|u|>|u'|$, then 
\begin{equation}
\frac{1}{|u-u'|}=\sum\limits_{k=0}^\infty\frac{|u'|^k}{|u|^{k+1}}P_k(\cos\theta),
\end{equation}
where $P_k$ is the $k$-th order Legendre polynomial and $\theta$ is the angle between $u, u'$. Use the Taylor series $(1+x)^\alpha=1+\alpha x+\frac{\alpha(\alpha-1)}{2}x^2+\cdots$, and the assumption that $|u_2|>>|u_1|$ we have
\begin{equation}\label{eq:legendre}
\begin{split}
\frac{\mu_1\mu_2}{|u_2-\nu_0u_1|^\alpha}&+\frac{\mu_0\mu_2}{|u_2+\nu_1u_1|^\alpha}=\frac{\mu_2(\mu_0+\mu_1)}{|u_2|^\alpha}+\\
&\frac{\mu_2(\mu_1\nu_0^2+\mu_0\nu_1^2)|u_1|^2}{|u_2|^{\alpha+2}}[\alpha P_2(\cos\theta)+\frac{\alpha(\alpha-1)}{2}P_1^2(\cos\theta)]+h.o.t.
\end{split}
\end{equation}
where $\theta$ is the angle between $u_1$ and $u_2$ and $h.o.t.$ stands for higher order terms. Let 
\begin{equation}
H_3=\frac{|v_2|^2}{2M_2}-u_2^TJv_2-\frac{\mu_2(\mu_0+\mu_1)}{|u_2|^\alpha},
\end{equation}
then $H_3$ is the Hamiltonian of the Kepler problem, where a fixed body of mass $\mu_2$ is located at the origin and another body of mass $\mu_0+\mu_1$ moves in a rotating frame and is attracted to the fixed body by the $\alpha$-potential. One can think of the fixed body as the sun and the other body as the union of the earth and moon.

The third and final assumption we shall make is that the center of mass of the earth-moon system moves on a nearly circular orbit about the sun. Since $H_3$ is the Hamiltonian of a Kepler problem in rotating coordinates, one of the circular orbits becomes a circle of critical points for $H_3$. In particular, $H_3$ has a critical point $u_2=a, v_2=-M_2Ja$ for any constant vector $a$ satisfying $|a|^{\alpha+2}=\alpha \mu_2$. Introduce coordinates, \[Z=\begin{pmatrix}u_2\\v_2\end{pmatrix}\] and a constant vector \[Z_0=\begin{pmatrix}a\\-M_2Ja\end{pmatrix}\] so that $H_3$ is a function of $Z$ and $\nabla H_3(Z_0)=0$. By Taylor's theorem 
\begin{equation}H_3(Z)=H_3(Z_0)+\frac{1}{2}(Z-Z_0)^TS(Z-Z_0)+O(|Z-Z_0|^3),\end{equation}
where $S$ is the Hessian of $H_3$ evaluated at $Z_0$. Since constants are lost in the formation of the equations of motion we shall ignore the constant $H_3(Z_0)$ in our further discussions. Thus we make the change of variables $Z-Z_0\to\epsilon^\beta X$, where $\beta$ is again a positive number to be chosen. 

So far, starting with (\ref{eq:H1}) we have proposed the following change of variables $u_1\to\epsilon^{2\delta}u_1$ and $Z-Z_0\to\epsilon^\beta X$. In order to have a symplectic change of variables (of multiplier $\epsilon^{2\beta}$) we must make the further change $v_1\to\epsilon^{2(\beta-\delta)}v_1$. Thus we propose the following symplectic change of variables in (\ref{eq:H1}):

\begin{equation}u_1\to\epsilon^{2\delta}u_1,\quad v_1\to\epsilon^{2(\beta-\delta)}v_1,\quad Z-Z_0\to\epsilon^\beta X. \end{equation}

The first restriction is that the kinetic energy and potential energy in $H_1$ should be of the same order of magnitude as mentioned in \cite{MeSc82}, and this leads to the restriction that $2\beta=(2-\alpha)\delta+\gamma$. The second restriction following Hill is that the second term in (\ref{eq:legendre}) should be of the same order of magnitude as the terms in $H_1$ and this leads to the condition $2\delta=\beta$. The smallest positive integer solutions for $\beta$ and $\delta$ are $\delta=1$, $\beta=2$ and $\gamma=\alpha+2$. With this choice of scale factors the Hamiltonian becomes
\begin{equation}\label{eq:H2}
\begin{split}
H=\frac{|v_1|^2}{2M_1}-u_1^TJv_1-\frac{\mu_0\mu_1}{|u_1|^\alpha}-\frac{\mu_2(\mu_1\nu_0^2+\mu_0\nu_1^2)|u_1|^2}{\alpha\mu_2}&[\alpha P_2(\cos\theta)+\frac{\alpha(\alpha-1)}{2}P_1^2(\cos\theta)]\\&+\frac{1}{2}X^TSX+O(\epsilon^2).
\end{split}
\end{equation}

In order to reduce the the number of constants in (\ref{eq:H2}) we shall make one further scaling of the variables. We introduce new variables $\xi$ and $\eta$ to eliminate the subscripts and use the fact that $P_1(x)=x$ and $P_2(x)=\frac{1}{2}(3x^2-1)$. Also we choose $a=((\alpha\mu_2)^{\frac{1}{\alpha+2}},0)$ so that the abscissa points at the sun. Make the symplectic change of coordinates (with multiplier $(\frac{\mu_0+\mu_1}{\alpha+2})^{\frac{2}{\alpha+2}}M_1$):
\begin{equation}
u_1=(\frac{\mu_0+\mu_1}{\alpha+2})^{\frac{1}{\alpha+2}}\xi, \quad v_1=(\frac{\mu_0+\mu_1}{\alpha+2})^{\frac{1}{\alpha+2}}M_1\eta, \quad X=(\frac{\mu_0+\mu_1}{\alpha+2})^{\frac{1}{\alpha+2}}M_1^{1/2}Y
\end{equation}
so that (\ref{eq:H2}) becomes
\begin{equation}
H=\frac{|\eta|^2}{2}-\xi^TJ\eta-\frac{\alpha+2}{|\xi|^\alpha}-\frac{1}{2}((\alpha+2)\xi_1^2-|\xi|^2)+\frac{1}{2}Y^TSY+O(\epsilon^2).
\end{equation}

Hill proposed to construct a lunar theory defined by the Hamiltonian 
\begin{equation}\label{eq:Hprime}
H'=\frac{|\eta|^2}{2}-\xi^TJ\eta-\frac{\alpha+2}{|\xi|^\alpha}-\frac{1}{2}((\alpha+2)\xi_1^2-|\xi|^2)
\end{equation}
and the equations defined by (\ref{eq:Hprime}) are known as Hill's lunar equations. Use the notation $\xi=(x,y)$ and $\eta=(p_x, p_y)$ for the position and momentum variables, then the effective potential 
\[V(x,y)=H'-\frac{1}{2}[(p_x+y)^2+(p_y-x)^2]=-\frac{\alpha+2}{2}x^2-\frac{\alpha+2}{(x^2+y^2)^{\frac{\alpha}{2}}}.\]


\end{document}